\documentclass{amsart}
\usepackage{amsmath}
\usepackage{amscd} 
\usepackage{amssymb} 
\usepackage{stmaryrd} 
\usepackage{times}
\usepackage{mathrsfs}

\usepackage{color}
\usepackage{booktabs}
\usepackage[all]{xy}
\def\lz{\par \vspace{8pt}}

\usepackage[colorlinks=true, pdfstartview=FitV, linkcolor=blue, pagebackref, citecolor=blue, urlcolor=blue]{hyperref}

\makeatletter
\@namedef{subjclassname@2020}{\textup{2020} Mathematics Subject Classification}
\makeatother

\newtheorem{ter}[equation]{Terminology}
\newtheorem{thm}[equation]{Theorem}
\newtheorem{lem}[equation]{Lemma}

\newtheorem{cor}[equation]{Corollary}
\newtheorem{prop}[equation]{Proposition}

\newtheorem{rem}[equation]{Remark}

\newtheorem*{thm*}{Theorem}
\newtheorem*{prop*}{Proposition}
\newtheorem*{cor*}{Corollary}
\newtheorem*{lem*}{Lemma}
\newtheorem*{MT*}{Main Theorem}

\newtheorem*{ques*}{Question}

\theoremstyle{definition} %

\newtheorem*{defn*}{Definition}

\newtheorem{eg}[equation]{Example}

\theoremstyle{remark} %

\newtheorem*{rmk*}{Remark}
\newtheorem*{rmks*}{Remarks}

\DeclareMathOperator{\Pf}{Pf}
\DeclareMathOperator{\Sim}{Sim}
\DeclareMathOperator{\Str}{Str}

\newcommand{\vssp}{\vspace{2,5pt}}

\DeclareMathOperator{\diag}{diag}

\DeclareMathOperator{\Cent}{Cent}

\DeclareMathOperator{\red}{red}

\newcommand{\Eins}{\mathrm{\mathbf{1}}}

\newcommand{\vph}{\varphi}
\newcommand{\eps}{\varepsilon}

\newcommand{\trans}{\mathsf{T}}

\newcommand{\bft}{\mathbf{t}}
\newcommand{\bfC}{\mathbf{C}}

\newcommand{\IN}{\mathbb{N}}
\newcommand{\ZZ}{\mathbb{Z}}
\newcommand{\IZ}{\ZZ}

\newcommand{\QQ}{\mathbb{Q}}
\newcommand{\IQ}{\QQ}

\newcommand{\msB}{\mathscr{B}}

\newcommand{\alg}[1]{{#1\text{-\textbf{alg}}}}
\newcommand{\kalg}{\alg{k}}

\newcommand{\oalg}{\mfo\mathchar45\mathbf{alg}}

\DeclareMathOperator{\Core}{Core}
\DeclareMathOperator{\Gal}{Gal}

\DeclareMathOperator{\Nil}{Nil}
\DeclareMathOperator{\Oct}{Oct}

\DeclareMathOperator{\ch}{char}

\newcommand{\bfAut}{\mathbf{Aut}}

\newcommand{\bfT}{\mathbf{T}}

\newcommand{\mfo}{\mathfrak{o}}
\newcommand{\mfp}{\mathfrak{p}}

\newcommand{\la}{\langle}
\newcommand{\ra}{\rangle}
\newcommand{\dla}{\langle\langle}
\newcommand{\dra}{\rangle\rangle}

\renewcommand{\:}{\colon\,}

\newcommand{\skiprootsystem}{\mathsf}
\newcommand{\rsF}{\skiprootsystem{F}_4}
\newcommand{\rsG}{\skiprootsystem{G}_2}

\newcommand{\rsD}{\skiprootsystem{D}}

\newcommand{\rsE}{\skiprootsystem{E}}

\DeclareMathOperator{\Her}{Her}

\DeclareMathOperator{\End}{End}

\DeclareMathOperator{\GL}{GL}
\numberwithin{equation}{section}

\begin{document}

\title{Homogeneous Freudenthal algebras and the first Tits construction}


\author[H.P. Petersson]{Holger P. Petersson}
\address{Petersson: Fakult\"at f\"ur Mathematik und Informatik, FernUniversit\"at in Hagen, D-58084 Hagen, Germany}
\email{holger.petersson@fernuni-hagen.de}


\author[M. Thakur]{Maneesh Thakur}
\address{Thakur: Indian Statistical Institute, Stat-Math.-Unit, 8th Mile Mysore Road, Bangalore-560059, India}
\email{maneesh.thakur@gmail.com}


\begin{abstract}
Freudenthal algebras over a field are basically the same as Jordan algebras of degree $3$ remaining simple under all base field extensions. These algebras are intimately linked, via their automorphism groups and structure groups, to simple algebraic groups over arbitrary fields. Our main concern  here will be the question of when these algebras are homogeneous in the sense that all their Jordan isotopes are isomorphic. We answer this question by presenting various necessary and sufficient conditions for homogeneity and by connecting it with the first Tits construction of cubic Jordan algebras, most notably through investigating Freudenthal division algebras over complete fields under a discrete valuation. We also study the first Tits construction in its own right by producing a local version of it, deriving a local-global principle, and by connecting it with the embeddibility of certain rank-$2$-tori into the automorphism group scheme of an Albert division algebra.
\end{abstract}

\maketitle

\setcounter{tocdepth}{1}
\tableofcontents

\section{Introduction}

In this paper, we will be concerned with the interplay of three important algebraic concepts: Freudenthal algebras, homogeneous Jordan algebras, and the first Tits construction.

To begin with, Freudenthal algebras form a class of algebraic structures that may be regarded as the analogue in degree $3$ of (unital) composition algebras. Just as composition algebras fit into the broader picture of alternative algebras whose elements satisfy a universal quadratic equation, Freudenthal algebras fit into the broader picture of Jordan algebras whose elements satisfy a universal cubic equation. And, while composition algebras exist only in dimensions $1,2,4,8$, the dimensions of Freudenthal algebras are confined to the numbers $1,3,6,9,15$, and $27$. Ignoring the low dimensions $1$ (the base field) and $3$ (the Jordan algebras of cubic \'etale algebras), Freudenthal algebras are twisted verstion of the Jordan algebra of $3$-by-$3$ hermitian matrices with entries in a split composition algebra and scalars down the diagonal, the latter condition being automatic if the base field has characteristic not $2$.

The similarity between composition algebras and Freudenthal algebras becomes most striking when considering the connection with exceptional algebraic groups. As in the case of octonion algebras and groups of type $\rsG$, passing from an algebra to its automorphism group scheme gives a one-to-one correspondence between Albert algebras, that is, Freudenthal algebras of dimension $27$, and groups of type $\rsF$. Other important connections, though not quite as close as the preceding one, exist between Albert algebras and groups of type $\rsD_4,\rsE_6,\rsE_7,\rsE_8$. In particular, the Kneser-Tits conjecture for groups of type $\rsE^{78}_{7,1}$ and $\rsE^{78}_{8,2}$ was settled by proving $R$-triviality for the structure group of an Albert division algebra \cite{MR4345011, MR4474885}.

To continue, homogeneity of Jordan algebras is a property intimately tied up with the notion of an isotope. Any invertible element $p$ of a Jordan algebra $J$ induces canonically  a new Jordan algebra structure on the underlying vector space, denoted by $J^{(p)}$ and called the $p$-isotope of $J$, whose identity element is the inverse of $p$ in $J$. Isotopes of $J$ will not always be isomorphic to $J$ but if they are, then $J$ is said to be homogeneous. The importance of isotopes, and of the notion of homogeneity,  derives from the fact that many useful properties attached to Jordan algebras hold only up to isotopy, that is, they may fail for the algebra itself but become valid  when passing to an appropriate isotope. For example, the question raised by Albert \cite{MR182647} as to whether every Albert division algebra contains a cyclic cubic subfield is open to this very day but has an affirmative answer in an appropriate isotope \cite{MR4242148}.

The two Tits constructions come into play when dealing with Freudenthal division algebras. Here it is the first construction (as the more elementary of the two), with its remarkable analogy to the Cayley-Dickson construction of composition algebras, that demands our attention. A first instance of its interplay with the preceding concepts may be found in \cite[Cor.~4.9]{MR734841} , which says that Albert division algebras arising from the first Tits construction are always homogeneous. It is a natural question to ask for the converse: is every homogeneous Albert division algebra a first Tits construction? Since, unfortunately, we have not been able to settle this question one way or the other, we decided to undertake a systematic investigation of homogeneous Freudenthal algebras. It is the purpose of this paper to present the results of our investigation.

In the next two sections, besides fixing notation and terminology, we recall some basic facts  about the subject matter that will be needed to understand the subsequent development of the paper. The emphasis is on (cohomological) invariants of Freudenthal algebras$\--$a tool we will make use of quite heavily later on. The question of which set of invariants classifies a specific class of Freudenthal algebras up to isomorphism (resp., up to isotopy) turns out to be particularly important. Along the way, we obtain  first characterizations of homogeneous Freudenthal algebras of dimension $9$ (Prop.~\ref{p.HOMDIS}). 

Section~$4$ is devoted to what could be called a local version of the first Tits construction for Albert division algebras. We derive a local-global principle in this setting (Thm.~\ref{embed}) and draw the connection with the embedibility of certain rank-$2$-tori in the automorphism group scheme of the underlying algebra (Thm.~\ref{normic}). 

In section~$5$, we begin by showing that, contrary to the case of Albert division algebras, nine-dimensional Freudenthal division algebras exist that are first Tits constructions but not homogeneous (Example~\ref{NINEFI}). We then proceed to derive necessary and sufficient conditions for reduced (that is, non-division) Freudenthal algebras to be homogeneous. As the main conclusion from these criteria, it will be shown that, loosely speaking, any sufficiently rich class of homogeneous Freudenthal algebras having dimension $< 27$ can always be enlarged to include higher-dimensional algebras with almost the same properties (Thm.~\ref{t.HOMEN}). In particular, if the Freudenthal algebra of symmetric $3$-by-$3$ matrices over a given field is homogeneous, then so are all Freudenthal algebras over that field (Thm.~\ref{t.HOMEN}~(c)). 

In section~$6$, we study Freudenthal algebras that are strictly homogeneous in the sense that they remain homogeneous under all base field extensions. Our approach leads to a complete classification of strictly homogeneous Freudenthal algebras (Cor.~\ref{c.STRIHO}): they are either split, or isomorphic to $D^{(+)}$ for some central associative division algebra $D$ of degree $3$, or an Albert division algebra arising from the first Tits construction.

In sections~$7$ and $8$, we take up the study of Freudenthal division algebras over a complete field under a discrete valuation. Such a study had been undertaken before for Albert division algebras only \cite{MR0379620}. Revisiting this study once more and extending its results to arbitrary Freudenthal division algebras makes sense for the following reasons. On the one hand, we obtain new examples of nine-dimensional Freudenthal division algebras that are first Tits constructions and homogeneous (Cors.~\ref{c.TRICHO}, \ref{RAFRHO}) while on the other, we are able to set the record straight with regard to the naive treatment of cubic forms adopted in \cite{MR0379620}. In developing the theory along the lines indicated, we take full advantage of the approach to the two Tits constructions presented in \cite[Chap.~VII]{MR4806163} and here, in particular, of the fact that this approach works not just over fields but, to a large extent, over arbitrary commutative base rings.

The paper concludes with a brief excursion into homogeneous Jordan algebras of Clifford type. We show, in particular, that the Jordan algebra of a pointed Pfister quadratic form is strictly homogeneous (Cor.~\ref{PFIHOM}).

\section{Proper Freudenthal algebras and the Tits constructions} \label{s.REDFRA} 
Throughout this paper, we fix a field $k$ of arbitrary characteristic. The category of unital commutative associative $k$-algebras will be denoted by $\kalg$. For notation, terminology and basic facts about (unital) composition algebras, (cubic) Jordan (in particular, Freudenthal) algebras and related algebraic groups, the reader is referred to \cite{MR4806163} and \cite{MR1632779}; when it comes to quadratic forms, our standard reference will be \cite{MR2427530}. In the following review, we will therefore confine ourselves to what is absolutely indispensable for understanding the subsequent development of the paper.

\subsection{Homogeneous Jordan algebras} \label{ss.HOJOAL} A \emph{Jordan $k$-algebra} is a vector space $J$ over $k$ together with a quadratic map $U\:J \to \End_k(J)$, $x \mapsto U_x$, (the \emph{$U$-operator}) and a distinguished quantity $1 = 1_J \in J$ (the \emph{unit element}) such that the identities
\[
U_1 = \Eins_J, \quad U_{U_xy} = U_xU_yU_x, \quad U_xV_{y,x} = V_{x,y}U_x
\]
hold in all scalar extensions, where $V_{x,y}z := U_{x,z}y$ for $x,y,z \in J$ and $U_{x,z} := U_{x+z} - U_x - U_z$ is the bilinearization of the $U$-operator. An element $p \in J$ is said to be \emph{invertible} if $U_p\:J \to J$ is bijective, in which case $p^{-1} := U_p^{-1}p$ is called the \emph{inverse} of $p$ in $J$. The set of invertible elements in $J$ will be denoted by $J^\times$. For $p \in J^\times$, the vector space $J$ together with the $U$-operator $x \mapsto U_x^{(p)} := U_xU_p$ and the unit element $1^{(p)} := p^{-1}$ is again a Jordan algebra, called the \emph{$p$-isotope} of $J$ and denoted by $J^{(p)}$. The isotopes $J^{(p)}$, $p \in J^\times$, are in general not isomorphic to $J$, but if they are, for all $p \in J^\times$, then $J$ is called \emph{homogeneous}. The \emph{autotopies} of $J$, that is, the isomorphisms from $J$ onto its various isotopes, form a subgroup of $\GL(J)$, denoted by $\Str(J)$ and called the \emph{structure group} of $J$, so that $J$ is homogeneous if and only if its structure group acts transitively on its invertible elements, hence the name.

\subsection{Cubic Jordan algebras} \label{ss.CUBJOR} Following \cite[34.1]{MR4806163}, we define a \emph{cubic Jordan algebra} over $k$ as a Jordan $k$-algebra $J$ together with a cubic form $N_J\:J \to k$ (the \emph{norm}), that is, a scalar polynomial law over $k$ that is homogeneous of degree $3$, satisfying the following conditions.
\begin{itemize}
\item  [(a)] $N_J$ \emph{permits Jordan composition} in the sense that the equations
\[
N_J(1_J) = 1, \quad N_J(U_xy) = N_J(x)^2N_J(y)
\]
hold strictly, i.e., in all base field extensions.

\item [(b)] For all field extensions $K/k$ and all $x \in J_K$ (the base change of $J$ from $k$ to $K$), the polynomial
\[
m_{J,x}(\bft) = N_J(\bft 1_{J_K} - x) = \bft^3 - T_J(x)\bft^2 + S_J(x)\bft - N_J(x)
\] satisfies the equations
\[
m_{J,x}(x) = (\bft m_{J,x})(x) = 0.
\]
\end{itemize}
Here $T_J\:J \to k$ (resp., $S_J\:J \to k$) is called the \emph{linear} (resp., \emph{quadratic}) \emph{trace} of $J$, which makes sense because $T_J$ (resp., $S_J$) is a linear (resp., quadratic) form. By contrast, the symmetric bilinear form $T_J\:J \times J \to k$ defined by
\begin{align*}
T_J(x,y) := T_J(x)T_J(y) - S_J(x,y) &&(x,y \in J)
\end{align*}
in terms of the linear trace and the bilinearized quadratic one is called the \emph{bilinear trace} of $J$. The \emph{adjoint} of $J$ is defined as the quadratic map
\[
J \longrightarrow J, \quad x \longmapsto x^\sharp := x^2 - T_J(x)x + S_J(x)1_J. 
\]
Note that if we know the adjoint and the bilinear trace of $J$, then we know its Jordan structure, thanks to the formula
\begin{align}
\label{UOPCUB} U_xy = T_J(x,y)x - x^\sharp \times y
\end{align}
for the $U$-operator, where $x \times y = (x + y)^\sharp - x^\sharp - y^\sharp$ denotes the bilinearized adjoint. Finally, we say that $J$ is \emph{regular} if its dimension is finite and the bilinear trace is a regular symmetric bilinear form.

Straightforward modifications of the preceding set-up lead to the notion of \emph{cubic alternative algebras}. We refer to \cite[34.11]{MR4806163} for details.

\subsection{Proper reduced Freudenthal algebras} \label{ss.SIREFR} By a \emph{Freudenthal algebra} over $k$ we mean a cubic Jordan $k$-algebra $J$ which either is isomorphic to $E^{(+)}$, for some cubic \'etale $k$-algebra $E$, or makes the base change $J_K$, for any field extension $K/k$, a simple Jordan $K$-algebra \cite[34.17, 39.8]{MR4806163}. A Freudenthal $k$-algebra is said to be \emph{proper} if it has dimension at least $6$; it is said to be \emph{reduced} if it contains an elementary frame, that is, a complete orthogonal system of idempotents that are \emph{elementary} in the sense that their adjoint is $0$ and their linear trace is $1$, see \cite[\S37]{MR4806163} for details.

By \cite[41.1]{MR4806163}, proper reduced Freudenthal algebras over $k$ up to isomorphism have the form $J := \Her_3(C,\Gamma)$, where $C$ is a composition algebra over $k$ and $\Gamma \in \GL_3(k)$ is a diagonal matrix: $\Gamma = \diag(\gamma_1,\gamma_2,\gamma_3)$ with $\gamma_1,\gamma_2,\gamma_3 \in k^\times$. Recall that $J$ consists of all $3$-by-$3$-matrices $x$ with entries in $C$ that are $\Gamma$-hermitian (i.e., $x = \Gamma^{-1}\bar x^\trans \Gamma$) and have scalars down the diagonal, the latter condition being automatic for $\mathrm{char}(k) \neq 2$. From \cite[36.4]{MR4806163} we deduce that norm, adjoint and bilinear trace of $J$ are given by the formulas
\begin{align*}
N_J(x) =\,\,&\xi_1\xi_2\xi_3 - \sum\gamma_j\gamma_l\xi_in_C(u_i) + \gamma_1\gamma_2\gamma_3t_C(u_1u_2u_3), \\
x^\sharp =\,\,&\sum\Big(\big(\xi_j\xi_l - \gamma_j\gamma_ln_C(u_i)\big)e_{ii} + \big(-\xi_iu_i + \gamma_i\overline{u_ju_l}\big)[jl]\big), 
\end{align*}
\begin{align}
\label{REDBLT} T_J(x,y) =\,\,&\sum\big(\xi_i\eta_i + \gamma_j\gamma_ln_C(u_i,v_i)\big)
\end{align}
for $x = \sum(\xi_ie_{ii} + u_i[jl]), y = \sum(\eta_ie_{ii} + v_i[jl]) \in J$, $\xi_i,\eta_i \in k$, $u_i,v_i \in C$, where we systematically adhere to the \emph{ternary cyclicity convention}: Unadorned sums like the above are to be taken over all cyclic permutations $(ijl)$ of $(123)$.

Recall from \cite[Exc.~37.22~(b)]{MR4806163} that up to isomorphism the diagonal entries of $\Gamma$ may be multiplied by arbitrary non-zero square factors in $k$ and that we may always assume $\gamma_3 = 1$. In particular, we put $\Her_3(C) := \Her_3(C,\Eins_3)$. We wish to understand the isotopes of $\Her_3(C)$. To this end, we require a harmless variation of \cite[Exc.~37.23]{MR4806163}.

\begin{prop} \label{p.JOMAST}
Let $C$ be a composition $k$-algebra, $J := \Her_3(C)$ and 
\[
\Gamma = \diag(\gamma_1,\gamma_2,\gamma_3) \in \GL_3(k). 
\]
With $p := \sum \gamma_ie_{ii} \in J^\times$, the assignment
\[
\sum(\xi_ie_{ii} + u_i[jl]) \longmapsto \sum\big((\gamma_i\xi_i)e_{ii} + (\gamma_j\gamma_lu_i)[jl]\big)
\]
for $\xi_i \in k$, $u_i \in C$, $1\le i\le 3$, defines an isomorphism $\vph$ from $\Her_3(C,\Gamma)$ onto the isotope $J^{(p^{-1})}$.
\end{prop}

\begin{proof} 
By \cite[Exc.~37.23]{MR4806163}, $\vph$ is an isomorphism from $\Her_3(C,\Gamma)^{(p)}$ onto $J$, hence also one from $\Her_3(C,\Gamma) = (\Her_3(C,\Gamma)^{(p)})^{(p^{-2})}$ to $J^{(\vph(p^{-2}))} = J^{(p^{-1})}$.
\end{proof}

\begin{cor} \label{c.JOMAIS}
Assume in Prop.~\emph{\ref{p.JOMAST}} that $C$ has dimension at most two. Then $\Her_3(C,\Gamma) \cong J$ if and only if  there exists $g \in \GL_3(C)$ such that $g p\bar g^\trans \in k\Eins_3$.
\end{cor}

\begin{proof}
By the proposition, $\Her_3(C,\Gamma) \cong J$ if and only if $J^{(p^{-1})} \cong J$, which happens if and only if some $\eta \in \Str(J)$ has $\eta(p) = 1_J$ \cite[Thm.~31.22~(d)]{MR4806163}. Since $J$, by the hypothesis on $C$, is the \emph{full} Jordan algebra of hermitian matrices with entries in $C$, the structure group of $J$ consists of the transformations $x \mapsto \gamma gx\bar g^\trans$, for $\gamma \in k^\times$, $g \in \GL_3(C)$ \cite[Thms.~6$\--$8]{MR407103}. The assertion follows.
\end{proof}

\begin{cor} \label{c.JOMACL}
Let $C$ be a composition $k$-algebra and $J := \Her_3(C)$. Up to isomorphism, the isotopes of $J$ are precisely of the form $\Her_3(C,\Gamma)$, for some diagonal matrix $\Gamma \in \GL_3(k)$.   
\end{cor}

\begin{proof}
If $C$ is regular, this is a special case of \cite[Thm.~40.10]{MR4806163}. Otherwise, \eqref{REDBLT} implies $C = k$ and $\ch(k) = 2$. Any $p \in J^\times$ may be viewed canonically as a ternary regular symmetric bilinear form. By a theorem of Albert \cite[Thm.~8]{MR1501952}, $p$ is either diagonalizable or alternating. But an alternating ternary symmetric bilinear form over $k$ cannot be regular. Hence $p$ is diagonalizable, and Prop.~\ref{p.JOMAST} can be applied to $\Gamma := p^{-1}$.
\end{proof}

Proper Freudenthal algebras that are not reduced are Jordan division algebras in the sense that all of their non-zero elements are invertible \cite[Thm.~39.6]{MR4806163}. Freudenthal division algebras exist only in dimensions $1,3,9,27$ \cite[Thm.~46.8]{MR4806163}. Examples may be found by means of the two Tits constructions \cite[Chap.~VII]{MR4806163}. The first Tits construction, which is the more easily accessible of the two, plays a particularly important role in the present investigation. Its main ingredients may be described as follows.

\subsection{The first Tits construction} \label{ss.FITICO} The input of the first Tits construction consists of a cubic alternative $k$-algebra $A$ and a scalar $\mu \in k$. The output is a cubic Jordan algebra $J = J(A,\mu)$ living on the direct sum $A \oplus Aj_1 \oplus Aj_2$ of three copies of $A$ as a vector space over $k$ under the cubic Jordan algebra structure with identity element, adjoint and norm given by
\begin{align}
1_J =\,\,&1_A = 1_A + 0\cdot j_1 + 0\cdot j_2, \notag \\
x^\sharp =\,\,& (x_0^\sharp - \mu x_1x_2) + (\mu x_2^\sharp - x_0x_1)j_1 + (x_1^\sharp - x_2x_0)j_2, \notag \\
\label{FITNOR} N_J(x) =\,\,&N_A(x_0) + \mu N_A(x_1) + \mu^2N_A(x_2) - \mu T_A(x_0x_1x_2)
\end{align}
for $x = x_0 + x_1j_1 + x_2j_2$, $x_0,x_1,x_2 \in A$. The bilinear trace of $J$ is then given by
\begin{align}
\label{FITBLT} T_J(x,y) =\,\,&T_A(x_0y_0) + \mu T_A(x_1y_2) + \mu T_A(x_2y_1),   
\end{align}
where $y = y_0 + y_1j_1 + y_2j_2$, $y_0,y_1,y_2 \in A$.

\begin{prop} \label{p.FITFRE}
For a cubic alternative $k$-algebra $A$ and $\mu \in k$, the following conditions are equivalent.
\begin{itemize}
\item [(i)] $J(A,\mu)$ is a regular Freudenthal algebra.

\item [(ii)] $J(A,\mu)$ is regular.

\item [(iii)] $A$ is regular and $\mu \in k^\times$.
\end{itemize}
\end{prop}

\begin{proof}
(i)$\Rightarrow$(ii). Obvious.

(ii)$\Rightarrow$(iii). Apply \eqref{FITBLT}.

(iii)$\Rightarrow$(i). Put $J := J(A,\mu)$ and note that condition (iii) is stable under base change, while condition (i) is stable under faithfully flat descent \cite[Cor.~39.32]{MR4806163}. Hence we may assume if necessary that $k$ is algebraically closed. By \cite[Exc.~42.21]{MR4806163}, therefore, it remains to consider the following cases. $\vssp$ \\
$1^\circ.$ $A = k$, $\ch(k) \ne 3$. Then $E = k[\bft]/(\bft^3 - \mu)$ is cubic \'etale over $k$ and $J \cong E^{(+)}$ is a Freudenthal algebra. $\vssp$ \\
$2^\circ.$ $A = k \times C$, where $C$ is a regular composition $k$-algebra. Then \cite[Exc.~42.27~(a)]{MR4806163} implies $J \cong \Her_3(C,\Gamma)$, $\Gamma = \diag(-1,-1,1)$, and this is a reduced Freudenthal algebra. \\
$3^\circ.$ $A$ is a central simple associative $k$-algebra of degree $3$. Then $J$ is an Albert algebra \cite[Cor.~42.15]{MR4806163}.
\end{proof}

There are two important properties of the first Tits construction that will play a crucial role in the present investigation. The first of these may be found in \cite[Cor.~46.12]{MR4806163}.

\begin{prop} \label{p.DIFITI}
For a cubic alternative $k$-algebra $A$ and $\mu \in k$, the first Tits construction $J(A,\mu)$ is a cubic Jordan division algebra if and only if $A$ is a division algebra and $0 \ne \mu \notin N_A(A^\times)$. \hfill $\square$    
\end{prop}

The second property we have in mind may be found in \cite[Thm.~IX.22]{MR0251099} for $\ch(k) \ne 2$ and in \cite[Thm.~8]{MR0271181} in general, see also \cite[Cor.~A4.5]{garibaldi2025solutionsexercisesbookalbert}.

\begin{thm} \label{t.KUMFIT}
Let $J$ be an Albert $k$-algebra and $A$ a central simple associative algebra of degree $3$ over $k$ such that $A^{(+)}$ is a subalgebra of $J$. Then there exists a $\lambda \in k^\times$ such that the inclusion $A^{(+)} \hookrightarrow J$  extends to an isomorphism from $J(A,\lambda)$ to $J$. \hfill $\square$  
\end{thm}

\subsection{The second Tits construction} \label{ss.SETICO} In its most general form, the second Tits construction doesn't show up in the present paper. Instead, it will be perfectly adequate to follow \cite[44.5]{MR4806163} and consider 
\begin{itemize}
\item an \emph{associative involutorial system} $\msB = (K,B,\tau)$
over $k$, consisting by definition of a composition $k$-algebra $K$ of dimension $r \in \{1,2\}$, a cubic associative $K$-algebra $B$ and a $K/k$-involution $\tau$ of $B$; and

\item an \emph{admissible scalar} for $\msB$, that is, a pair $(p,\mu) \in J_0^\times \times K^\times$ ($J_0 := H(\msB)$, the Jordan $k$-algebra of $\tau$-symmetric elements in $B$), satisfying $N_B(p) = n_K(\mu)$
\end{itemize}
as input. The output will then be a cubic Jordan $k$-algebra $J = J(\msB,p,\mu)$ living on the direct sum $J_0 \oplus Bj$ of $J_0$ and $B$ as a vector space over $k$ under the cubic Jordan algebra structure with identity element, adjoint and norm given by
\begin{align}
1_J =\,\,&1_B = 1_B + 0\cdot j, \notag \\
x^\sharp =\,\,&\big(x_0^\sharp - up\tau(u)\big) + \big(\bar\mu\tau(u^\sharp)p^{-1} - x_0u\big)j, \notag \\
\label{SETNOR} N_J(x) =\,\,&N_{J_0}(x_0) + \mu N_B(u) + \bar\mu\overline{N_B(u)} - T_{J_0}\big(x_0,up\tau(u)\big)
\end{align}
for $x = x_0 + uj$, $x_0 \in J_0$, $u \in B$. The bilinear trace of $J$ is then given by
\begin{align}
\label{SETBLT} T_J(x,y) =\,\,&T_{J_0}(x_0,y_0) + T_B\big(up\tau(v)\big) + T_B\big(vp\tau(u)\big),   
\end{align}
where $y = y_0 + vj$, $y_0 \in J_0$, $v \in B$. In this context, we refer to $\msB$, more specifically, as an associative involutorial system \emph{of the $r^{th}$ kind} and call $\Core(\msB) := K$ its core.

Though the second Tits construction is much more delicate to handle than the first, there are close connections between the two. For example, if $\msB$ as above is of the second kind, then $J_K \cong J(B,\mu)$ is a first Tits construction \cite[Cor.~44.20]{MR4806163}. Since $K$ is faithfully flat over $k$, this and Prop.~\ref{p.FITFRE} imply

\begin{prop}
Let $\msB = (K,B,\tau)$ be an associative involutorial system of the second kind over $k$ and $(p,\mu)$ an admissible scalar for $\msB$. The second Tits construction $J(\msB,p,\mu)$ is a regular Freudenthal algebra over $k$ if and only $B$ is regular over $K$. \hfill $\square$    
\end{prop}

The two important properties we have singled out for the first Tits construction have natural analogues for the second as well, which may be found in \cite[Thm.~46.10]{MR4806163} (resp., \cite[Exc.~IX.12.5]{MR0251099} and \cite[Thm.~9]{MR0271181}, see also \cite[Thm.~A4.4]{garibaldi2025solutionsexercisesbookalbert}).

\begin{prop} \label{p.DISETI}
If $\msB = (K,B,\tau)$ is an associative involutorial system of the second kind over $k$ and $(p,\mu)$ is an admissible scalar for $\msB$, then $J(\msB,p,\mu)$ is a cubic Jordan division $k$-algebra if and only if $H(\msB)$ is and $\mu \notin N_B(B^\times)$. \hfill $\square$
\end{prop}

\begin{thm} \label{t.ETSETI}
Let $J$ be an Albert $k$-algebra and $(B,\tau)$ a central simple associative algebra of degree $3$ with unitary involution over $k$ such that $H(B,\tau)$ is a subalgebra of $J$. With $K := \Cent(B)$ abd $\msB := (K,B,\tau)$, there is an admissible scalar $(p,\mu)$ for $\msB$ such that the inclusion $H(B,\tau) \hookrightarrow J$ can be extended for an isomorphism from $J(B,\tau,p,\mu) := J(\msB,p,\mu)$ onto $J$. \hfill $\square$   
\end{thm}

\section{Invariants} \label{s.INVAR} In this section, following various sources in the literature, we attach invariants to any proper Freudenthal algebra over $k$. The following concept will be of crucial importance.

\subsection{Reduced models and the coordinate dimension} \label{ss.REDMOD} Let $J$ be a proper Freudenthal $k$-algebra. Following \cite{MR1417849}, a \emph{reduced model} of $J$ is defined as a reduced Freudenthal algebra $J_{\red}$ over $k$ such that $J_K \cong (J_{\red})_K$ for all field extensions $K/k$ making the base change $J_K$ reduced over $K$. Reduced models of $J$ always exist and are unique up to isomorphism \cite[Thm.~2.8]{MR1417849}. But note that this result is non-trivial only in dimensions $9$ and $27$ because proper Freudenthal algebras in the complementary dimensions $6$ and $15$ are reduced, hence their own reduced models. For any proper Freudenthal $k$-algebra $J$, its reduced model $J_{\red}$ has  a certain coordinate composition algebra, unique up to isomorphism, whose dimension will henceforth be referred to as the \emph{coordinate dimension} of $J$. 

\subsection{The invariants $\bmod\,2$} \label{ss.INMOTW} Let $C$ be a composition $k$-algebra of dimension $2^r$, $0 \le r\le 3$, $\Gamma = \diag(\gamma_1,\gamma_2,\gamma_3) \in \GL_3(k)$ and $J = \Her_3(C,\Gamma)$ the corresponding proper reduced Freudenthal $k$-algebra. Then
\begin{align}
\label{BOTINV} f_r(J) := n_C    
\end{align}
is an $r$-Pfister quadratic form and an isotopy invariant of $J$ that classifies proper reduced Freudenthal algebras up to isotopy \cite[Thm.~41.8]{MR4806163}; it is called the \emph{$r$-invariant $\bmod\,2$ of $J$}.

If $J$ is regular, we put
\begin{align}
\label{TOPINV} f_{r+2}(J) := n_C \perp Q_J, \quad Q_J := \la \gamma_2\gamma_3,\gamma_3\gamma_1,\gamma_1\gamma_2\ra \otimes n_C,  
\end{align}
which is an $(r+2)$-Pfister quadratic form and also an invariant of $J$ \cite[Lemma~41.7]{MR4806163}, called its \emph{$(r+2)$-invariant $\bmod\,2$}. Note for $\gamma_3 = 1$ that
\begin{align}
\label{TOPPFI} f_{r+2}(J) = \dla -\gamma_1,-\gamma_2\dra \otimes n_C = \la 1,\gamma_1,\gamma_2,\gamma_1\gamma_2\ra \otimes n_C,  
\end{align}
which is indeed a Pfister quadratic form. For $\ch(k) \ne 2$, the invariants $\bmod\,2$ may be interpreted
Galois cohomologically (via the Arason invariant) as elements $f_{r+i}(J) \in H^{r+i}(k,\IZ/2)$, $i = 0,2$. In the notation of \cite[41.11]{MR4806163} we have $f_r(J) = \Pf_J$, $f_{r+2}(J) = \Pf_J^{+2}$, allowing us to conclude from \cite[Thm.~41.21]{MR4806163} that \emph{the invariants $\bmod\,2$ are classifying invariants for regular proper reduced Freudenthal algebras}.

If $J$ is an arbitrary regular proper Freudenthal $k$-algebra of coordinate dimnsion $2^r$ as in \ref{ss.REDMOD}, we recover the invariants $\bmod\,2$ by passing to the reduced model:
\begin{align}
\label{JAYRET} f_r(J) := f_r(J_{\red}), \quad f_{r+2}(J) := f_{r+2}(J_{\red}).  
\end{align}
But in this more general setting, the invariants $\bmod\,2$ are no longer classifying and, in fact, there is another important invariant, the invariant $\bmod\,3$, that has to be called in for help. We do so by discussing the dimensions $9$ and $27$ separately.

\subsection{Invariants: Freudenthal algebras of dimension $9$} \label{ss.INMOTH} There is a close connection between Freudenthal algebras of dimension $9$ and central simple associative algebras of degree $3$ with unitary involution. More specifically, we deduce from \cite[Exc.~55.12]{MR4806163} and \cite[Exc.~55.12]{garibaldi2025solutionsexercisesbookalbert} that a cubic Jordan $k$-algebra $J$ is a Freudenthal algebra of dimension $9$ if and only if there exists a central simple associative $k$-algebra $(B,\tau)$ 
of degree $3$ with unitary involution \cite[44.23, Exc.~9.33]{MR4806163} such that
\begin{align}
\label{FRENIN} J \cong H(B,\tau) = \{u \in B \mid \tau(u) = u\}
\end{align}
as cubic Jordan $k$-algebras; in this case, $(B,\tau)$ is unique up to isomorphism.

Realizing $J$ by means of \eqref{FRENIN}, the invariants $\bmod\,2$ as defined in \ref{ss.INMOTW} fit into this picture as follows.
\begin{itemize}
\item $f_1(J)$ is the isomorphism class of $K := \Cent(B)$, the centre of $B$, a quadratic \'etale $k$-algebra. $K$ being a $k$-form of the split quadratic \'etale $k \times k$, we may view $f_1(J)$ also as an element $f_1(J) \in H^1(k,\IZ/2)$. Since composition algebras over fields are classified by their norms, we can identify $f_1(J)$ equally well with the norm of $K$ and thus with a Pfister $1$-form over $k$, in agreement with \ref{ss.INMOTW}. 

\item $f_3(J)$ is the Pfister $3$-form determined by $\tau$ as defined in \cite[(19.4)]{MR1632779}; one may also regard it as an octonion algebra over $k$, written as $\Oct(\tau)$, a point of view adopted in \cite[Exc.~46.23~(a)]{MR4806163}. 
\end{itemize}
We define the \emph{invariant $\bmod\,3$ of $J$} as 
\begin{itemize}
\item the class $g_2(J)$ of $B$ in the Brauer group of $K$. If $K$ is a field, then we may view $g_2(J)$ as an element $g_2(J) \in H^2(k,\bf{\mu}_{[K]})$ \cite[(30.16)]{MR1632779}.
\end{itemize}
According to \cite[Thm.~(30.21)]{MR1632779} and \cite[Thm.~2.4]{MR2063796}, $f_1,f_3,g_2$ are classifying invariants of $9$-dimensional Freudenthal algebras $J$ over $k$. Note that $J$ is a division algebra if and only if $g_2(J) \ne 0$..

\subsection{Isotopes and distinguished involutions} \label{ss.ISDINV} If $J$ is a Freudenthal $k$-algebra of dimension $9$, then so are all its isotopes. More precisely,  if $J \cong H(B,\tau)$ for some central simple associative $k$-algebra of degree $3$ with unitary involution, then $J^{(p)} \cong H(B,\tau^p)$ for any invertible element $p \in J$ \cite[31.15]{MR4806163}, where $\tau^p$, $u \mapsto p^{-1}\tau(u)p$ is the $p$-twist of $\tau$ \cite[10.8]{MR4806163}. It follows that $f_1,g_2$ are \emph{isotopy invariants} of nine-dimensional Freudenthal algebras, i.e., remain unaffected by passing to isotopes. Setting $K := \Cent(B)$, we also recall that the assignment $p \mapsto \tau^p$ gives a surjection from $J^\times$ to the set of $K/k$-involutions on $B$ \cite[Thm.~8.7.4]{MR770063}. Also, the $K/k$-involutions of $B$ are classified by the invariant $f_3$ (\cite[Thm.~(19.6)]{MR1632779}, \cite[Thm.~2.4]{MR2063796}).

A $K/k$-involution $\tau$ of $B$ is said to be \emph{distinguished} if the $3$-Pfister quadratic form $f_3(J)$ is hyperbolic (\cite[(19.8)]{MR1632779}, \cite[2.6]{MR2063796}), equivalently, the corresponding octonion algebra is split. Distinguished $K/k$-involutions always exist on $B$ provided $B$ admits $K/k$-involutions at all (\cite[(19.12)]{MR1632779}, \cite[Thm.~2.10]{MR2063796}). Hence the results assembled above yield the following conclusion.

\begin{prop} \label{p.HOMDIS}
Let $(B,\tau)$ be a central simple associative $k$-algebra of degree $3$ with unitary involution and put $K := \Cent(B)$. Then the following conditions are equivalent.
\begin{itemize}
    \item[(i)] The nine-dimensional Freudenthal $k$-algebra $H(B,\tau)$ is homogeneous. 	
\item[(ii)] All $K/k$-involutions of $B$ are conjugate.  
\item[(iii)] All $K/k$-involutions of $B$ are distinguished.
\item[(iv)] $H(B,\tau)^{\times}=k^{\times}\{b\tau(b)|b\in B^{\times}\}$.
\item[(v)] Every Albert $k$-algebra containing $H(B,\tau)$ as a subalgebra is a first Tits construction.
\item[(vi)] Every reduced Albert $k$-algebra containing $H(B,\tau)$ as a subalgebra is split.

\end{itemize}
\end{prop}
\begin{proof} The equivalence of (i), (ii) and (iii) is clear from the above. Setting $J_0 := H(B,\tau)$, it remains to settle the following implications. 

(i)$\Leftrightarrow$(iv). By \cite[Thms.~6,7]{MR407103}, the structure group of $J_0$ consists of the transformations $x \mapsto \lambda bx\tau(b)$ for some $\lambda \in k^\times$, $b \in B^\times$. Hence (iv) is equivalent to the orbit of $1_B$ under the action of the structure group being all of $J_0^\times$, hence to $J_0$ being homogeneous.

(iii)$\Rightarrow$(v). Assume (iii) and let $A$ be an Albert algebra over $k$ containing $H(B,\tau)$ as a subalgebra. Then, by Thm. \ref{t.ETSETI}, we can write $A=J(B,\tau, u, \mu)$ for a suitable admissible scalar $(u,\mu)$. Now $f_3(A)=n_C$, where $C$ is the octonion algebra of $H(B,\tau^v)$, $v := u^{-1}$ \cite[Thm.~1.8]{MR1337184}. 
By (iii), $\tau^v$ is distinguished, so $C$ must be split, forcing $A$ to be a first Tits construction.


(v)$\Rightarrow$(vi). Every reduced Albert algebra that is a first Tits construction is split \cite[Cor.~4.2]{MR734841}.

(vi)$\Rightarrow$(iii). Assume $u \in B$ has norm $1$. Then $A := J(B,\tau,v,1)$, $v = u^{-1}$, by Prop.~\ref{p.DIFITI} is a reduced Albert algebra containing $H(B,\tau)$ as a subalgebra. By (vi), therefore, $A$ is split. Hence so is its octonion algebra, which agrees with the octonion algebra of $\tau^u$. Thus $\tau^u$ is distinguished. If $u \in B^\times$ is arbitrary, then $w := N_B(u)^{-1}u^3 = N_B(u)^{-1}U_uu$ has norm $1$ and $J_0^{(w)} \cong J_0^{(u)}$. Thus $\tau^w$ is distinguished and conjugate to $\tau^u$, forcing $\tau^u$ to be distinguished as well.  
\end{proof}
This proposition turns out to be particularly useful in connection with the following result, which may be found in \cite[Thm.~(19.14)]{MR1632779} and \cite[Thm.~3.1]{MR2063796}.

\begin{thm} \label{t.CHADIS}
Let $(B,\tau)$ be a central simple associative $k$-algebra of degree $3$ with unitary involution. The involution $\tau$ is distinguished if and only if  the Freudenthal $k$-algebra $H(B,\tau)$ is a first Tits construction: there exist a cubic \'etale $k$-algebra $E$ and a scalar $\mu \in k^\times$ such that $H(B,\tau) \cong J(E,\mu)$. \hfill $\square$	
\end{thm}

\subsection{Invariants: Albert algebras} \label{ss.INVALB} In analogy to Freudenthal algebras of dimension $9$, Albert algebras, that is, Freudenthal algebras of dimension $27$, have three important invariants that allow  cohomological interpretations if certain low characteristics are excluded. Let $A$ be an Albert $k$-algebra.
\begin{itemize}
\item  $f_3(A)$ and $f_5(A)$ are the invariants $\bmod\,2$ of $A$ as defined in \ref{ss.INMOTW} for $r = 3$. As in the case of nine-dimensional Freudenthal algebras, $f_3(A)$ can be interpreted as an octonion algebra, denoted by $\Oct(A)$ and called the \emph{octonion algebra of $A$}. We recall from \cite[Prop.~(40.5)]{MR1632779} and \cite[Thm.~4.10]{MR2063796} that $\Oct(A)$ is split if and only if $A$ is a first Tits construction.

\item $g_3(A)$, the \emph{invariant $\bmod\,3$ of $A$}, belongs to $H^3(k,\IZ/3)$ provided $\ch(k) \ne 3$. For a precise definition of this invariant, see \cite[p.~537]{MR1632779}, \cite{MR99j:17045,PR:char3}. A particularly important property of $g_3$ is that it characterizes division algebras: \emph{$A$ is an Albert division algebra if and only if $g_3(A) \ne 0$.} \cite[Thm.~40.8]{MR1632779}, \cite[Thm.~3.2]{MR99j:17045}, \cite[Thm.~7]{PR:char3}
\end{itemize}
We note that $f_3$ and $g_3$ are isotopy invariants, while $f_5$ is not \cite[Cor.~60]{MR3916096}. If $k$ has characteristic not $2$ or $3$, then $f_3,f_5,g_3$ are cohomological invariants of Albert algebras and, basically, there are no others \cite[Prop.~8.6]{G:lens}, \cite[Thm.~22.5]{GMS}. It is an open question raised by Serre \cite{MR1321649} whether Albert algebras are classified by their invariants.

\section{Subalgebras arising from the first Tits construction} \label{s.SUFITI} In this section, we study nine-dimensional subalgebras of Albert division algebras that arise from the first Tits construction. For convenience, we formalize the situation as follows.

\subsection{The first Tits construction at a separable cubic subfield} \label{ss.FITSEP} Let $A$ be an Albert division $k$-algebra and $L \subseteq A$ a separable cubic subfield. $A$ is said to be a \emph{first Tits construction at $L$} if there exists a $\lambda \in k^\times$ such that the inclusion $L \hookrightarrow A$ can be extended to an embedding from the first Tits construction $J(L,\lambda)$ into $A$. In view of this terminology, our first result of this section may be regarded as a local-global principle for first Tits constructions.

\begin{thm} \label{embed}
An Albert division $k$-algebra is a first Tits construction if and only if it is a first Tits construction at every separable cubic subfield of $A$.    
\end{thm}
For the proof we need an important result from \cite{MR4242148}:
\begin{thm} \label{t.CYCIST} Let $A$ be an Albert division algebra over 
$k$. Then there is $v\in A^\times$ such that the isotope $A^{(v)}$ is \emph{cyclic}, i.e., contains a cyclic cubic  subfield. \hfill $\square$
\end{thm}
\begin{cor}\label{embed-cor} Let $A$ be a homogeneous Albert division algebra over $k$. Then $A$ is cyclic. 

\hfill $\square$
\end{cor}

\noindent \emph{Proof of Thm.~\ref{embed}}. The ``only if''-direction has been established in \cite[Cor.~4.5]{MR734841}. To prove the converse, we assume $A$ is a first Tits construction at every separable cubic subfield and must show that $A$ is a first Tits construction. Let $v\in A$ be an element that is not a scalar, i.e., $v\notin k$. Let $L=k[v]$ be the subalgebra of $A$ generated by $v$. By \cite[Prop.~46.6]{MR4806163}, $L$ is a cubic subfield of $A$. Assume first that $L$ is separable. Let $w:=v^{-1}$ and $R_w:L\rightarrow L$ denote the right homothety by $w$. Then $R_w(v)=1$. Since $A$ is a first Tits construction at $L$, there exists $\lambda\in k^\times$ such that $L \subseteq S := J(L,\lambda)\subseteq A$. Let $\chi:S\rightarrow S$ be the map given by  
\[
\chi(l_0 + l_1j_1 + l_2j_2) = wl_0 + (l_1w)j_1 + (w^{-1}l_2w^\sharp)j_2
\]
for $l_0, l_1, l_2\in L$. Since $\chi\in \Str(S)$ by \cite[Exc.~42.25]{MR4806163}, we conclude from \cite[Thm.~4.4]{MR4009439} that $\chi$ can be extended to  some element $\widetilde{\chi} \in \Str(A)$. In other words, $\widetilde{\chi}\in \Str(A)$ is such that $\widetilde{\chi}(v)=1$. It follows that $A^{(v)}\cong A$. If $L$ is inseparable, then $k$ has characteristic $3$ and $T_A(v)=T_A(v^\sharp)=0$. Hence $U_vv = v^3=\alpha.1_A$ for some $\alpha\in k^\times$. It follows again that $A^{(v)}\cong A$. Hence $A$ is homogeneous and therefore, by Cor. \ref{embed-cor}, it follows that $A$ contains a cubic cyclic subfield $M$. $A$ being a first Tits construction at $M$ yields $\nu\in k^\times$ such that $M \subseteq J(M,\nu)\subseteq A$. But $M$ is cyclic, so with a generator $\sigma$ of its Galois group, we may form the cyclic $k$-algebra $D := (M/k,\sigma,\nu)$, which is a central associative division algebra of degree $3$ over $k$ such that $D^{(+)} \cong J(M,\nu)$ \cite[Prop.~5.1]{MR87h:17038b}. Hence $D^{(+)}$ is a subalgebra of $A$ and $A$ is a first Tits construction. \hfill $\square$ \lz

\begin{eg} \label{e.ALGEMA}
\textbf{Generic matrices.} Considering the Albert algebra of generic matrices over an integeral domain as defined in \cite{MR1699633} and passing to its central closure, we obtain an Albert division algebra $A$ over an appropriate field $F$ that is a pure second Tits construction \cite[Thm.~1]{MR1699633}. In fact, it is as far removed from the first Tits construction as it could possibly be. We claim: \emph{$A$ does not contain any separable cubic subfield (over $F$) at which it is a first Tits construction.} Indeed, if $L \subseteq A$ is a separable cubic subfield, the quadratic trace of $A$ restricted to $L^\perp$, the orthogonal complement of $L$ relative t the bilinear trace, is anisotropic on the one hand \cite[Thm.~2]{MR1699633}, and up to a sign the Scharlau transfer induced by the linear trace of the Springer form of $L$ on the other \cite[Exc.~42.16~(a)]{MR4806163}. Hence the Springer form of $L$ is anisotropic as well, and the assertion follows from \cite[Thm.~4.4]{MR734841}.\hfill $\square$  
\end{eg}

If $A$ is an Albert division algebra not arising from the first Tits construction, Thm.~\ref{embed} implies that $A$ is not a first Tits construction at some separable cubic subfields of $A$. But it may very well be one at others. It is therefore of interest to find criteria spelling out necessary and sufficient conditions for $A$ to be a first Tits construction at individual separable cubic subfields. One of the conditions we are going to present below relates to certain rank-$2$-tori in the algebraic group $\bfAut(A)$ of automorphisms of $A$. The following lemma will pave the way. 

Recall that for a central simple algebra $(B,\sigma)$ with a unitary $K/k$ involution $\sigma$, $K$ being the center of $B$, maximal $k$-tori in ${\bf SU}(B,\sigma)$ are precisely of the form $T={\bf SU}(E,\sigma)$, where $E\subseteq B$ is a maximal commutative $\sigma$-stable \'{e}tale $K$-subalgebra (see \cite{MR1081540}, 3.4).

\begin{lem}\label{rational} Let $(B,\sigma)$ be a central associative division algebra of degree $3$ with unitary involution over $K/k$, where $K$ is the center of $B$. Let $\text{\bf T}$ be the maximal $k$-torus $\text{\bf SU}(E,\sigma)/\text{\bf Z}\subseteq \text{\bf PSU}(B,\sigma)$, where $E\subseteq B$ is a $\sigma$-stable separable cubic extension of $K$ and ${\bf Z}$ is the center of $\text{\bf SU}(B,\sigma)$. Let $g\in T:=\text{\bf T}(k)$ be represented by $p\in{\bf SU}(E,\sigma)$. Then $p^3\in \text{\bf SU}(E,\sigma)(k)$. 
\end{lem}
\begin{proof} Let $\Gamma = \text{Gal}(k_s/k)$, where $k_s$ is a separable closure of $k$. Since $\text{\bf T}$ is defined over $k$ and $g\in T$, we have $\gamma(g)=g$ for all $\gamma\in\Gamma$. Let $g\in T$ be represented by $p\in\text{\bf  SU}(E,\sigma)$, i.e., $g=p\text{\bf Z}$. 
 It follows that $p^{-1}\gamma(p)\in\text{\bf Z}$ for all $\gamma\in\Gamma$. 
 Since $\text{\bf Z}$ is a $3$-torsion abelian group, it follows that $p^3$ is fixed by $\Gamma$ and hence $p^3\in\text{\bf SU}(E,\sigma)(k)$.  
\end{proof}
We recall from \cite[p.~291 and Prop.~(18.24)]{MR1632779} that the discriminant of an \'etale $k$-algebra $L$ is a quadratic \'etale algebra over $k$, denoted by $\Delta(L)$.  By an isogenic embedding of algebraic groups we mean a homomorphism with finite (central) kernel. With this in mind, we can now prove:

\begin{thm}\label{normic} Let $A$ be an Albert division algebra over $k$ and $L \subseteq A$ a separable cubic subfield. Writing $\text{\bf L}^{(1)}$ for the torus of norm-$1$-elements in $L$, the following conditions are equivalent.
\begin{itemize}
\item [(i)]$A$ is a first Tits construction at $L$.

\item [(ii)] The discriminant of $L$ is a subalgebra of $\Oct(A)$.


\item [(iii)] The rank-$2$-torus $\text{\bf L}^{(1)}$ embeds in $\bfAut(A)$ isogenically over $k$.
\end{itemize}
\end{thm}

\begin{proof} (i)$\Leftrightarrow$(ii). Write $\Delta$ for the discriminant of $L$. By \cite[Thm.~4.4]{MR734841}, (i) holds if and only if the Galois closure of $L$ is a splitting field of $A$. It therefore suffices to show that this latter condition is equivalent to $\Delta$ being a subalgebra of $\Oct(A)$. Assume first that $L$ is cyclic. Then $\Delta$ is split, and $L$ agrees with its own Galois closure. Now, $L$ is a splitting field of $A$ iff $A_L$ is split iff so is $\Oct(A_L) \cong \Oct(A)_L$ \cite[Thm.~4.10]{MR2063796} iff $\Oct(A)$ is split (by Springer's theorem \cite[Cor.~18.5]{MR2427530} since $[L:k]$ is odd) iff $\Oct(A)$ contains $\Delta$ (since $\Delta$ is split). On the other hand, if $L$ is not cyclic, then $\Delta$ is a field, and $L_\Delta$ is the Galois closure of $L$. Now, $L_\Delta$ splits $A$ iff it splits $A_\Delta$ over $\Delta$ iff $\Oct(A_\Delta) \cong \Oct(A)_\Delta$ is split (by the previous case) iff $\Delta$ is a subfield of $\Oct(A)$.

(i)$\Rightarrow$(iii). Assume $J(L,a)\subseteq A$ for some $a\in k^{\times}$. 
    From \eqref{FRENIN} we conclude $J(L,a)\cong H(B,\sigma)$ for some central simple associative $k$-algebra $(B,\sigma)$ of degree $3$ with unitary involution, while Thm.~\ref{t.ETSETI} implies $A\cong J(B,\sigma, u, \mu)$ for some admissible scalar $(u,\mu)$ relative to $(B,\sigma)$. We have $J(L,a)= L \oplus Lj_1\oplus Lj_2$ and, by (\cite{MR1632779}, 29.16, 37.B), 
    $$\text{\bf Aut}(J(L,a))^{\circ}=\text{\bf PGU}(B,\sigma)=
    \text{\bf SU}(B,\sigma)/\text{\bf Z},$$ where $\text{\bf Z}$ is the center of $\text{\bf SU}(B,\sigma)$ (see \cite[23.4]{MR1632779}). 
For $R \in \kalg$ and $p \in \textbf{L}^1(R)$, the map
\[
\phi(R)_p\:J(L,a)_R \longrightarrow J(L,a)_R,\; x_0 + x_1j_{1R} + x_2j_{2R} \longmapsto x_0 + (x_1p^{-1})j_{1R} + (px_2)j_{2R}
\]
for $x_0,x_1,x_2 \in L_R$ fixes the unit element and by \eqref{FITNOR} preserves norms, hence is an automorphism of $J(L,a)_R$ \cite[Exc.~34.18]{MR4806163}. Thus we have an embedding of $\text{\bf L}^{(1)}$ in $\text{\bf Aut}(J(L,a))$ over $k$.  Maximal $k$-tori in $\text{\bf Aut}(J(L,a))$ are of the form $\text{\bf SU}(E,\sigma)/\text{\bf Z}$ for $E\subseteq B$, a $\sigma$-stable cubic \'{e}tale subalgebra of $B$ containing $K$ and $\text{\bf Z}$ the center of $\text{\bf SU}(B,\sigma)$. It follows that $\text{\bf L}^{(1)}\cong\text{\bf SU}(E,\sigma)/\text{\bf Z}$ for a suitable $\sigma$-stable $E\subseteq B$ as above. Let $\phi:\text{\bf L}^{(1)}\cong\text{\bf SU}(E,\sigma)/\text{\bf Z}$ be an isomorphism defined over $k$. Recall that the center $\text{\bf Z}$ of $\text{\bf SU}(B,\sigma)$ is $3$-torsion.
	
	Let $\phi(p)=a_p\text{\bf Z}$, where $a_p\in\text{\bf SU}(E,\sigma)$. By Lemma \ref{rational}, $a_p^3\in\text{\bf SU}(E,\sigma)(k)$. By Lemma \ref{rational}, the map $\eta_p:A=J(B,\sigma, u, \mu)\longrightarrow A$ given by $\eta_p(b,x)=(a_p^3ba_p^{-3}, a_p^3x)$ is an automorphism of $A$ defined over $k$. It follows that the map $\text{\bf L}^{(1)}\longrightarrow\text{\bf Aut}(A)$ given by $p\mapsto \eta_p$ is a well defined isogenic embedding over $k$. 
    
  (iii)$\Rightarrow$(i). This is proved in \cite[Thm.~4.12]{MR3795634}. If the rank-$2$-torus $\text{\bf L}^{(1)}$ embeds in $\mathrm{\bf Aut}(A)$ isogenically over $k$, then the discriminant of $L$ is a subalgebra of $\mathrm{Oct}(A)$. Now use (ii)$\Rightarrow$(i). 
	\end{proof}

\begin{cor} \label{c.NECOCT}
For an octonion division $k$-algebra $C$ to be isomorphic to the octonion algebra of some Albert division algebra over $k$ it is necessary that there exist a separable quadratic field extension of $k$ that is not isomorphic to a subalgebra of $C$.     
\end{cor}

\begin{proof}
Let $A$ be an Albert division algebra over $k$ having $C \cong \Oct(A)$. Arguing indirectly, let us assume that every separable quadratic field extension of $k$ is isomorphic to a subalgebra of $C$. By Thm.~\ref{normic}, $A$ is a first Tits construction at every separable cubic subfield. By Thm.~\ref{embed}, therefore, $A$ is a first Tits construction, forcing $C \cong \Oct(A)$ to be split, a contradiction.    
\end{proof}

\begin{rem} \label{r.OCNECO}
\emph{The necessary condition spelled out in the preceding corollary is, of course, not sufficient. Let $C$ be the unique octonion division algebra over $\IQ$, the rationals \cite[Cor.~23.23]{MR4806163}. Albert division algebras over $\IQ$ having $C$ as associated octonion algebra do not exist, for the simple reason that all rational Albert algebras are reduced \cite[Cor.~46.17]{MR4806163}, and yet, (separable) quadratic field extensions of $\IQ$ not isomorphic to any subalgebra of $C$ exist in abundance: all real quadratic number fields.} 
\end{rem}

In Thm.~\ref{normic}, we can do without the assumption that $L$ be a subalgebra of $A$. More precisely, we have the following:
\begin{cor}\label{f5} Let $A$ be an Albert division algebra over $k$. Then $f_5(A)=0$ if and only if there is an isogenic $k$-embedding $\text{\bf L}^{(1)}\rightarrow\text{\bf Aut}(A)$, where $L$ is a cubic separable field extension of $k$.
\end{cor}
\begin{proof} If $f_5(A)=0$, then we can write $A=J(B,\sigma, u, \mu)$, as a second Tits construction, with $\sigma$ a distinguished unitary involution of $B$ and $(u,\mu)$ an admissible scalar \cite[Thm.~4.4]{MR2063796}. By (\cite[Them. 3.1]{MR2063796}), there exists a cubic separable field extension $L/k$, $L\subseteq H(B,\sigma)$ with discriminant $K = \Cent(B)$. By Thm. ~\ref{normic}, we have a $k$-isogenic $k$-embedding $\text{\bf L}^{(1)}\rightarrow\text{\bf Aut}(A)$. The converse is proved in 
\cite{MR3795634}.
\end{proof}
\begin{rem}
\emph{In \cite[Thm.~4.3]{MR3795634}, it was shown that if $A$ is an Albert algebra over $k$ such that $\text{\bf Aut}(A)$ allows an isogenic $k$-embedding ${\bf L}^{(1)}\rightarrow{\bf Aut}(A)$ for a cubic cyclic
field extension $L/k$, then $f_3(A)=0$, i.e., $A$ is a first construction. The preceding corollary can be interpreted as a generalization of this. To be precise, if we have an isogenic embedding ${\bf L}^{(1)}\hookrightarrow\text{\bf Aut}(A)$ over $k$, where $L$ is a cyclic cubic field extension of $k$, then $A\otimes_k L$ is split. To see this note that since $L$ is cyclic, the torus ${\bf L}^{(1)}$ becomes split when extending scalars from $k$ to $L$, so $\bfAut(A_L)$ contains a split torus of rank $2$. But a group of type $F_4$ containing a split torus of rank $2$ must be necessarily split (see \cite{MR224710}, Table II of Tits indices). Thus $\bfAut(A_L)$ is split, and so is $A_L$. Hence, by \cite[Thm.~2]{MR73558}, $L$ embeds in an isotope $A'$ of $A$ and, by Thm.~\ref{normic}, the discriminant of $L$ is a subalgebra of $\Oct(A')$. Since $\text{Oct}(A)$ is unchanged by passing to an isotope of $A$, it follows that $\text{Oct}(A)$ too contains the discriminant of $L$ as a subalgebra. But since $L$ is cyclic over $k$, its discriminant is split, hence $\text{Oct}(A)$ splits, so $A$ is a first construction.}
\end{rem}

\section{Criteria for homogeneity} \label{redal} 
We begin our list of criteria by recalling again, but this time on a more formal level, the following result of \cite[Cor.~4.9]{MR734841}.

\begin{thm} \label{t.FITIHO}
Every Albert division $k$-algebra arising from the first Tits construction is homogeneous. \hfill $\square$    
\end{thm}
\noindent We will see in a moment that the analogue of this theorem for Freudenthal division algebras of dimension $9$ does not hold. In order to prove this, we combine \eqref{FRENIN} with the existence of distinguished involutions in \ref{ss.ISDINV} and with Thm.~\ref{t.CHADIS} to obtain

\begin{prop} \label{p.FREDIV}
Every Freudenthal division $k$-algebra of dimension $9$ has an isotope that is a first Tits construction. \hfill $\square$   
\end{prop}

\begin{eg} \label{NINEFI}
(a) Let $J$ be a Freudenthal division $k$-algebra which is not a first Tits construction. By Prop.~\ref{p.FREDIV}, neither $J$ nor any of its isotopes is homogeneous. In particular, there are Freudenthal division algebras of dimension $9$ that are first Tits constructions but not homogeneous. $\vssp$ \\
(b) How do we find Freudenthal division algebras of dimension $9$ that do not arise from the first Tits construction? In order to clarify this, we return to the Albert division algebra $A$ over the field $F$ in Example~\ref{e.ALGEMA} and pick a Freudenthal subalgebra $J \subseteq A$ of dimension $9$ \cite[Thm.~45.11]{MR4806163}. Given any separable cubic subfield $L \subseteq J$, $A$ is not a first Tits construction at $L$ (Example~\ref{e.ALGEMA}), so in particular, there is no $a \in F^\times$ having $J \cong J(L,a)$. Thus $J$ is not a first Tits construction. 
\end{eg}

\noindent We now proceed to further criteria for Freudenthal algebras to be homogeneous. In order to phrase the main result that we are going to obtain along the way in an adequate manner, a conceptual preparation will be needed.

\subsection{Quasi-homogeneous Freudenthal algebras} \label{ss.QUAFRE} A proper Freudenthal $k$-algebra $J$ is said to be \emph{quasi-homogeneous} if its invariants do not change when passing to an isotope. Writing $2^r$ for the co-ordinate dimension of $J$ and noting that $f_r(J)$ as well as $g_2(J)$ for $r = 1$, $g_3(J)$ for $r = 3$ are isotopy invariants, $J$ is quasi-homogeneous if and only if $f_{r+2}(J^{(p)}) = f_{r+2}(J)$ for all $p \in J^\times$.

Homogeneous proper Freudenthal algebras are clearly quasi-homogeneous. Conversely, since proper Freudenthal algebras that are reduced or have dimension $< 27$ are classified by their invariants, they are quasi-homogeneous if and only if they are homogeneous. On the other hand, since we do not know whether Albert algebras are classified by their invariants, we do not know, either, whether quasi-homogeneous Albert algebras that are not homogeneous really exist; if they do, they must be division algebras.  \lz

\noindent With this terminology, we can now state the main result of this section.

\begin{thm} \label{t.HOMEN}
\emph{(a)} If all Freudenthal $k$-algebras of dimension $15$ are homogeneous, then all Albert algebras over $k$ are quasi-homogeneous. $\vssp$ \\
\emph{(b)} If all reduced Freudenthal $k$-algebras of dimension $9$ are homogeneous, then so are all Freudenthal algebras of dimension $\ge 9$ over $k$. $\vssp$ \\
\emph{(c)} If $\Her_3(k)$ is homogeneous, then so are all Freudenthal algebras over $k$.
\end{thm}
\lz 

\noindent The proof of this theorem rests on a number of criteria spelling out necessary and/or sufficient conditions for individual Freudenthal algebras to be (quasi-)homogeneous. 

\begin{prop} \label{p.HOINTW}
A proper Freudenthal $k$-algebra of coordinate dimension $2^r$ is quasi-homogeneous if and only if $f_{r+2}(J^{(p)})$ is hyperbolic for all $p \in J^\times$.    
\end{prop}

\begin{proof}
Let $J$ be a proper Freudenthal algebra over $k$ that is quasi-homogeneous. It suffices to show that $f_{r+2}(J^{(p)})$ is hyperbolic for some $p \in J^\times$. If $J$ is reduced, this follows from \cite[Thm.~41.26]{MR4806163}. If $J$ is a division algebra of dimension $9$ given in the form \eqref{FRENIN}, this follows from the existence of distinguished involutions (see \ref{ss.ISDINV}). And finally, if $J$ is an Albert division algebra, this follows from \cite[Thm.~4.7]{MR2063796}.    
\end{proof}

\begin{prop} \label{p.HOREOR}
\emph{(a)} Let $J$ be a Freudenthal $k$-algebra of dimension $9$. If $J_{\red}$ is homogeneous, then $J$ is homogeneous. $\vssp$ \\
\emph{(b)} Let $A$ be an Albert $k$-algebra. If $A_{\red}$ is homogeneous, then $A$ is quasi-homogeneous.
\end{prop}

\begin{proof}
We treat both cases simultaneously by considering a Freudenthal $k$-algebra $J$ of coordinate dimension $2^r$, $r = 1,3$. We must show that $f_{r+2}(J^{(p)}) = f_{r+2}(J)$ for all $p \in J^\times$. Since $f_r$ is an isotopy invariant, we may apply \eqref{JAYRET} and obtain
\[
f_r\big((J^{(p)})_{\red}\big) = f_r(J^{(p)}) =f_r(J) = f_r(J_{\red}).
\]
But $f_r$ is a \emph{classifying} isotopy invariant of reduced proper Freudenthal algebras. It follows that $(J^{(p)})_{\red}$ and $J_{\red}$ are isotopic, hence isomorphic since $J_{\red}$ is homogeneous by hypothesis. This implies 
\[
f_{r+2}(J^{(p)}) = f_{r+2}\big((J^{(p)})_{\red}\big) = f_{r+2}(J_{\red}) = f_{r+2}(J),
\]
as desired.
\end{proof}
\lz

\noindent In a way, Prop.~\ref{p.HOREOR} reduces the study of arbitrary proper homogeneous Freudenthal algebras to that of reduced ones. For the rest of this section, we feel therefore justified to restrict our attention to this special case. \lz

\noindent We denote by $I$ the fundamental ideal in the Witt ring $W(k)$, i.e., the ideal generated by the Pfister bilinear forms $\langle 1, -a \rangle = \dla a\dra$, for $a\in k^\times$. For a regular quadratic form $q$ over $k$, 
\begin{align*}
\mathrm{Ann}_{W(k)}(q):=\{c\in W(k)|c.q = 0\;\text{in the Witt group of $k$}\},    
\end{align*}
is the \emph{annihilator of $q$} in $W(k)$. We can now prove:  
 

\begin{prop} \label{p.redfr}
If $C$ is a regular composition algebra over $k$ and $J := \Her_3(C)$, then the following conditions are equivalent.
\begin{itemize}
\item [(i)] $J$ is homogeneous.

\item [(ii)] For all $\gamma,\delta \in k^\times$, the Pfister quadratic form $\langle\langle\gamma,\delta\rangle\rangle \otimes n_C$ is hyperbolic.

\item [(iii)] For all $\gamma\in k^\times$, the Pfister quadratic form $\dla \gamma\dra \otimes n_C = n_C \oplus (-\gamma)n_C$ is universal.

\item[(iv)] $I^2 \subseteq \mathrm{Ann}_{W(k)}(n_C)$.
\end{itemize}
\end{prop}

\begin{proof}
Let $\dim_k(C) = 2^r$ and note by Cor.~\ref{c.JOMACL} that the isotopes of $J$ up to isomorphism are precisely of the form $\Her_3(C,\Gamma)$ for some diagonal matrix $\Gamma \in \GL_3(k)$. Hence (i)$\Leftrightarrow$(ii) follows from Prop.~\ref{p.HOINTW} and \eqref{TOPPFI}. For $\gamma,\delta \in k^\times$ and $q := \dla\gamma\dra \otimes n_C$ we have $\dla\gamma,\delta\dra \otimes n_C = \dla\delta\dra \otimes n_C$. Hence (ii)$\Leftrightarrow$(iii) follows from \cite[Prop.~9.8]{MR2427530}. And, finally, (ii)$\Leftrightarrow$(iv) follows from the fact that $I^2$ is generated by $\dla\gamma,\delta\dra$, $\gamma,\delta \in k^\times$.
\end{proof}

\begin{cor} \label{c.FREHOM}
Split Freudenthal algebras of dimension other than $6$ are homogeneous.    
\end{cor}

\begin{proof}
Freudenthal algebras  of dimension $1$ or $3$ are homogeneous even if they are not split. We are thus reduced to considering split Freudenthal algebras of dimension at least $9$, so let $J = \Her_3(C)$ with $C$ a regular split composition algebra of dimension at least $2$. Since $n_C$ is hyperbolic, condition (ii) of Prop.~\ref{p.redfr} holds trivially, forcing $J$ to be homogeneous. 
\end{proof}

\begin{prop} \label{p.QUAHOM}
For a quaternion $k$-algebra $B$, the following conditions are equivalent. 
\begin{itemize}
\item [(i)] $\Her_3(B)$ is homogeneous.

\item [(ii)] The norm of any octonion algebra over $k$ containing $B$ as a subalgebra is universal.
\end{itemize}
In this case, any Albert algebra over $k$ whose octonion algebra contains $B$ as a subalgebra is quasi-homogeneous.
\end{prop}

\begin{proof}
The norm of any octonion algebra as in (ii) has the form $\dla\delta\dra \otimes n_B$ for some $\delta \in k^\times$. Hence (i)$\Leftrightarrow$(ii) follows from Prop.~\ref{p.redfr} for $C = B$. In the final statement, let $A$ be an Albert $k$-algebra such that $B \subseteq C := \Oct(A)$. Since $C \cong \Oct(A_{\red})$ , Prop.~\ref{p.HOREOR}~(b) allows us to assume that $A$ is reduced. Since $n_C$ is universal by (ii), so is $\dla\delta\dra \otimes n_C$ for all $\delta \in k^\times$, and the assertion follows from Prop.~\ref{p.redfr}.   
\end{proof}

\begin{prop} \label{p.ETAHOM}
For a quadratic \'etale $k$-algebra $K$, the following conditions are equivalent.
\begin{itemize}
\item [(i)] $\Her_3(K)$ is homogeneous.

\item [(ii)] There are no octonion division $k$-algebras containing $K$ as a subalgebra.

\item [(iii)] The norm of any quaternion $k$-algebra containing $K$ as a subalgebra is universal.
\end{itemize}
In this case, $\Her_3(C)$ is homogeneous for any composition algebra $C$ over $k$ containing $K$ as a subalgebra.
\end{prop}

\begin{proof}
(i)$\Leftrightarrow$(ii). The norm of any octonion algebra $C$ over $k$ containing $K$ as a subalgebra has the form $n_C \cong \dla\gamma,\delta\dra \otimes n_K$ for some $\gamma,\delta \in k^\times$. By Prop.~\ref{p.redfr}, therefore, any such $C$ is split if and only if (i) holds.

(i)$\Leftrightarrow$(iii). The norm of any quaternion algebra $C$ over $k$ containing $K$ as a subalgebra has the form $n_C \cong \dla\gamma\dra \otimes n_K$ for some $\gamma \in k^\times$. By Prop.~\ref{p.redfr}, $n_C$ for any such $C$ is universal if and only if (i) holds.

In the final statement, assume first that $C$ is an octonion algebra. By (ii), $C$ is split, hence so is $\Her_3(C)$ and thus, in particular, $\Her_3(C)$ is homogeneous (Cor.~\ref{c.FREHOM}). On the other hand, if $C$ is a quaternion algebra, then $n_C$ is universal by (iii) and hence so is $\dla\gamma\dra \otimes n_C$, for any $\gamma \in k^\times$. Thus, again, $\Her_3(C)$ is homogeneous by Prop.~\ref{p.redfr}.
\end{proof}

\begin{prop} \label{p.BASHOM}
If $k$ has characteristic not two, then $\Her_3(k)$ is homogeneous if and only if there are no quaternion division algebras over $k$. 
\end{prop}

\begin{proof}
Since $\ch(k) \ne 2$, we identify  quadratic and symmetric bilinear forms over $k$ in the usual way. For any $\gamma,\delta \in k^\times$, the quadratic form $\dla\gamma,\delta\dra \otimes n_k$ is the norm of the quaternion algebra $(\gamma,\delta)$, hence hyperbolic if and only if $(\gamma,\delta)$ is split. Hence the assertion follows from Prop.~\ref{p.redfr} for $C := k$, which is a \emph{regular} composition $k$-algebra by our hypothesis on $k$. 
\end{proof}

Not only the preceding proof breaks down in characteristic two but also the result itself. This will follow from the next proposition and the subsequent remark.

\begin{prop} \label{p.BATWOM}
If $k$ has characteristic two, then $\Her_3(k)$ is homogeneous if and only if $k$ is perfect. In this case, there are no quaternion division algebras over $k$.    
\end{prop}

\begin{proof}
If $k$ is perfect, then $k^2 = k$. Since the diagonal entries of any diagonal matrix $\Gamma \in \GL_3(k)$ can be changed by arbitrary non-zero square factors without changing the isomorphism class of $\Her_3(k,\Gamma)$, we conclude $\Her_3(k,\Gamma) \cong J := \Her_3(k)$, so $J$ is homogeneous. If $k$ is not perfect, then there exists $\gamma \in k^\times \setminus k^{\times 2}$, and \cite[Example~39.35]{MR4806163}    with $\Gamma := \diag(1,1,\gamma)$ shows that $\Her_3(k,\Gamma)$ is not isomorphic to $J$. Thus $J$ is not homogeneous. The final statement follows from \cite[Exc.~19.27]{MR4806163}.
\end{proof}

\begin{rem}
\emph{The converse of the final statement in Prop.~\ref{p.BATWOM} does not hold: let $F$ be an algebraically closed field of characteristic two and $k := F((\bft))$ the field of formal Laurent series in the variable $\bft$ over $F$. Though $k$ is not perfect, hence $\Her_3(k)$ is not homogeneous by Prop.~\ref{p.BATWOM}, there are no quaternion division algebras over $k$ \cite[Props.~1, 2]{MR364380}.}   
\end{rem}

\noindent \emph{Proof of Thm.~}\ref{t.HOMEN}. (a) Let $A$ be an Albert $k$-algebra, pick any quaternion subalgebra $B \subseteq \Oct(A)$ and apply Prop.~\ref{p.QUAHOM}. 

(b) By Prop.~\ref{p.HOREOR}~(a), all Freudenthal $k$-algebras of dimension $9$, reduced or division, are homogeneous. Next let $B$ be a quaternion algebra and $A$ an Albert algebra over $k$. Pick any quadratic \'etale subalgebra $K \subseteq B$ in the first case, $K \subseteq C := \Oct(A)$ in the second. By hypothesis, $\Her_3(K)$ is homogeneous, so Prop.~\ref{p.ETAHOM} implies that $\Her_3(B)$ is homogeneous as well while $C$ is split. But this implies that $A$ is a first Tits construction, hence homogeneous by Thm.~\ref{t.FITIHO}.

(c) By (b) it suffices to show that $\Her_3(K)$ is homogeneous for all quadratic \'etale $k$-algebras $K$. We have to show $\Her_3(K,\Gamma) \cong J := \Her_3(K)$ for any diagonal matrix $\Gamma = \diag(\gamma_1,\gamma_2,\gamma_3) \in \GL_3(k)$. From the hypothesis we conclude $\Her_3(k,\Gamma) \cong \Her_3(k)$, so Cor.~\ref{c.JOMAIS} implies $gp\bar g^\trans = gpg^\trans \in k\Eins_3$ for some $g \in \GL_3(k)$, $p := \sum \gamma_ie_{ii} \in \Her_3(k) \subseteq J$. Reading this in $K$ rather than $k$ and applying Cor.~\ref{c.JOMAIS} again, the assertion follows. \hfill $\square$

\begin{eg} \label{e.redal} It has already been observed in \cite{MR734841} and now follows again from Prop.~\ref{p.redfr} that a regular proper reduced Freudenthal algebra whose co-ordinate norm is universal is homogeneous. \emph{The converse, however, does not hold, not even in the special case of Abert algebras.}	
	
In order to see this, let $k_n := \bfC((\bft_1,\dots,\bft_n))$ for any positive integer $n$ be the field of iterated formal Laurent-series in the variable $\bft_1,\dots,\bft_n$ with complex coefficients.  By the main results of \cite{MR364380}, there exist a ramified quadratic extension over $k_1$, a ramified quaternion algebra over $k_2$, and a ramified octonion algebra over $k_3$. Hence we find an unramified octonion algebra $C$ over $k_4$ whose norm, therefore, is not universal. On the other hand, Springer's theory \cite{MR70664} of residue forms over local fields implies that anisotropic quadratic forms over $k_n$ for any $n$ have dimension at most $2^n$. In particular, the $5$-Pfister form $f_5(A)$ with $A := \Her_3(C)$ is isotropic, hence hyperbolic, hence zero in $W(k_4)$, forcing $A$ to be homogeneous.
\end{eg}

\section{Strictly homogeneous proper Freudenthal algebras} \label{stricthom}
We define a Jordan algebra $J$ over $k$ to be \emph{strictly homogeneous} if $J\otimes_kL$ is homogeneous for all field extensions $L$ of $k$. In this section we take up the study of such algebras. By Thm.~\ref{t.FITIHO}, first Tits construction Albert algebras are examples of such. In order to deal with the question of which proper Freudenthal algebras are strictly homogeneous, we require an elementary technical result.

\begin{prop}\label{key}
	Let $q:V \to k$ be an anisotropic quadratic form over $k$ and suppose ${\bf s},{\bf t}$	are independent indeterminates. Then the quadratic form $q_K \oplus {\bf s} q_K$ over the field $K := k({\bf s},{\bf t})$ does not represent the element ${\bf t} \in K$.
\end{prop}

\begin{proof} We always write $q$ instead of $q_K$ for simplicity. Arguing indirectly, there are elements  $x,y \in V_K$ such that $q(x) + {\bf s} q(y) = {\bf t}$. Then $x \ne 0$ or $y \ne 0$.	Writing
	\begin{align*}
		x = \frac{u({\bf s},{\bf t})}{g({\bf s},{\bf t})}, \quad y = \frac{v({\bf s},{\bf t})}{g({\bf s},{\bf t})}
	\end{align*}
	with $u({\bf s},{\bf t}),v({\bf s},{\bf t}) \in V[{\bf s},{\bf t}] := V \otimes_k k[{\bf s},{\bf t}]$ not both zero and $0 \ne g({\bf s},{\bf t}) \in k[{\bf s},{\bf t}]$, we may clear denominators to conclude
	\begin{align}
		\label{ENCE} q\big(u({\bf s},{\bf t})\big) + {\bf s} q\big(v({\bf s},{\bf t})\big) = {\bf t} g({\bf s},{\bf t})^2.
	\end{align}
	For some $m,n,p \in \mathbb{N}$, we have the expansions
	\begin{align*}
		u({\bf s},{\bf t}) =\,\,&u_m({\bf s}){\bf t}^m + \cdots, \\
		v({\bf s},{\bf t}) =\,\,&v_n(s){\bf t}^n + \cdots, \\
		g({\bf s},{\bf t}) =\,\,&g_p({\bf s}){\bf t}^p + \cdots,  
	\end{align*}
	where $\cdots$ refers to lower terms in ${\bf t}$, $u_m({\bf s}), v_n({\bf s}) \in V[{\bf s}]$, $0 \ne g_p({\bf s}) \in k[{\bf s}]$ and $u_m({\bf s}) \ne 0$ (resp. $v_n({\bf s}) \ne 0$) if $u({\bf s},{\bf t}) \ne 0$ (resp. $v({\bf s},{\bf t}) \ne0$). We conclude
	\begin{align*}
		q\big(u({\bf s},{\bf t})\big) =\,\,&q\big(u_m({\bf s})\big){\bf t}^{2m} + \cdots, \\
		q\big(v({\bf s},{\bf t})\big) =\,\,& q\big(v_n({\bf s})\big){\bf t}^{2n} + \cdots, \\
		g({\bf s},{\bf t})^2 =\,\,&g_p({\bf s})^2{\bf t}^{2p} + \cdots,
	\end{align*}
	where $q(u_m({\bf s}))$ (resp. $q(v_n({\bf s}))$) is non-zero whenever $u({\bf s},{\bf t})$ (resp. $v({\bf s},{\bf t})$) is since $q_{k({\bf s})}$ continues to be anisotropic. Plugging this into \eqref{ENCE}, we deduce
	\begin{align}
		\label{HIGHTE} q\big(u_m({\bf s})\big){\bf t}^{2m} + \cdots + {\bf s} q\big(v_n({\bf s})\big){\bf t}^{2n} + \cdots = g_p({\bf s})^2{\bf t}^{2p+1} + \cdots.
	\end{align}
	Here the degree in ${\bf t}$ of the right-hand side is odd, so we will arrive at a contradiction once we have shown that the degree in ${\bf t}$ of the left-hand side is even. This is clear if $m \ne n$. On the other hand, if $m = n$, then the coefficient of the highest term in ${\bf t}$ of the left-hand side is 
	\[
	q\big(u_m({\bf s})\big) + {\bf s} q\big(v_n({\bf s})\big),
	\]
	which cannot be zero unless both summands are non-zero. But then it is the sum of two non-zero polynomials in ${\bf s}$, the first summand being of even, the second summand being of odd degree. Hence the sum itself is not zero, and we have shown in all cases that the left-hand side of \eqref{HIGHTE} has even degree in ${\bf t}$. 
\end{proof}

We have, as an immediate consequence of the above proposition, 

\begin{thm}\label{red-str-homog} Let $J$ be a proper reduced Freudenthal algebra  over $k$. Then $J$ is strictly homogeneous if and only if $J$ is split of dimension at least $9$.
\end{thm}
\begin{proof} The property of $J$ being split of dimension at least $9$ is stable under base field extensions, hence Cor.~\ref{c.FREHOM}
forces $J$ to be strictly homogeneous. Conversely, let this be so. We first assume that $J$ is regular. Let $C$ denote the coordinate algebra of $J$ and $n_C$ its norm. For all field extensions $K/k$ and all $\gamma \in K^\times$, the Pfister quadratic form $n_C \otimes K \perp \gamma(n_C \otimes K)$ is universal over all field extensions $L/K$ (Prop.~\ref{p.redfr}). By Prop.~\ref{key}, therefore, $n_C$ must be isotropic over $k$, hence $C$ must be split of dimension at least $2$. It remains to discuss the case that $J$ is not regular. Then $k$ has characteristic $2$ and, up to isotopy, $J = \Her_3(k)$. We must show that $J$ is not strictly homogeneous. Since there are field extensions of $k$ that are not perfect, this follows from Prop.~\ref{p.BATWOM}.
\end{proof}

\begin{cor} \label{c.STRIHO}
\emph{(a)} A Freudenthal algebra of dimension $9$ over $k$ is strictly homogeneous if and only if it is isomorphic to $A^{(+)}$, for some central simple associative $k$-algebra $A$ of degree $3$. $\vssp$ \\
\emph{(b)} An Albert algebra over $k$ is strictly homogeneous if and only if it is a first Tits construction. 
\end{cor}

\begin{proof}
The ``if''-direction is clear in (a), and follows from Thm.~\ref{t.FITIHO} in (b). Conversely, let $J$ be any strictly homogeneous Freudenthal $k$-algebra of dimension $9$ (resp., $27$). If $J$ is reduced, then it is split by Thm.~\ref{red-str-homog}, and (a) (resp., (b)) holds. On the other hand, if $J$ is a division algebra, pick any separable cubic subfield $L \subseteq J$. Then the extended algebra $J_L$ over $L$ is still strictly homogeneous but also reduced, hence split by what we have just seen. If $J$ has dimension $9$, we may write $J = H(B,\tau)$ for some central simple associatve $k$-algebra $(B,\tau)$ with unitary involution. Since $H(B,\tau)$ becomes split after extending scalars from $k$ to $L$, so does the centre of $B$, which therefore, being quadratic \'etale over $k$, must have been split to begin with. Thus $(B,\tau) \cong (A \times A^{\mathrm{op}},\eps)$ for some central simple associative algebra $A$ of degree $3$ over $k$, $\eps$ being the exchange involution. Thus $J \cong A^{(+)}$, as claimed in (a). If $J$ has dimension $27$, put $C := \Oct(A)$. Since $J_L$ is split, so is $C_L$, hence $C$ since the degree of $L$  over $k$ is odd. But then $A$ must be a first Tits construction, as claimed in (b).
\end{proof}

\begin{eg}\label{homog-not-str-homog}
Let $F=\mathbb{Q}(\sqrt{2})$ and $(B_0,\sigma)$ be a degree $3$ central division algebra over $F$ with an involution $\sigma$ of second kind over $F/\mathbb{Q}$, these always exist \cite[Chap.~5,~Prop.~1]{MR340440}. Let $J_0:=H(B,\sigma)$. Then $J_0$ is a Freudenthal division algebra over $\mathbb{Q}$ of dimension $9$. Let $B=B_0\otimes_{\mathbb{Q}}\mathbb{Q}(i)$ and $\tau=\sigma\otimes 1$. Then $(B,\tau)$ is a central division algebra over $K:=\mathbb{Q}(\sqrt{2})(i)$ with an involution $\tau$ of the second kind over $K/k$, where $k=\mathbb{Q}(i)$. Consider $J:={J_0}\otimes_{\mathbb{Q}}k$. Then $J=H(B,\tau)$ is a $9$-dimensional Freudenthal division algebra over $k$. Since $k$ has no real embeddings, any octonion algebra over $k$ must be split \cite[Cor.~23.21]{MR4806163}. Hence the involution $\tau$ on $B$ must be distinguished, as is any involution of second kind on $B$. Hence $J$ is homogeneous Freudenthal division algebra over $k$ which, however, thanks to Cor.~\ref{c.STRIHO}, is not strictly homogeneous. 
\end{eg}

\section{Freudenthal division algebras and valuations} \label{s.FREDIVA} Throughout the next two sections, we fix a field $F$ that is complete under a discrete valuation. We begin by  collecting a number of technicalities that will be necessary for establishing the main results in the penultimate section of the paper. 

\subsection{Complete fields} \label{ss.COMFIE} Let us begin by fixing some notation. The discrete valuation belonging to $F$ will be denoted ba $\lambda$, so $\lambda\:F\to \IZ \cup \{\infty\}$ is a surjective map satisfying the usual properties; for basic facts about discrete valuations, see \cite{MR2215492}. We denote by $\mfo := \{\alpha \in F \mid \lambda(\alpha) \ge 0\}$ the valuation ring of $F$ ($\lambda$ always being understood), which is a local PID, by $\mfp := \{\alpha \in F \mid \lambda(\alpha) > 0\}$ the valuation ideal, i.e., the unique maximal ideal, of $\mfo$, and by $\bar F := \mfo/\mfp$ the residue field of $F$. The natural map from $\mfo$ to $\bar F$ will be indicated by $\alpha \mapsto \bar\alpha$.

\subsection{Cubic Jordan division algebras over complete fields} \label{ss.CUJOCO} Let $J$ be a cubic Jordan division algebra over $F$, always assumed to be finite-dimensional. Note that $J$ has either degree $3$ or dimension $1$. Free use will be made of the valuation theory for Jordan division rings developed in \cite{MR0347922}. In particular, by \cite[Satz~5.1]{MR0347922}, $\lambda$ has a unique extension to a discrete valuation $\lambda_J\:J \to \IQ \cup \{\infty\}$ given by 
\begin{align}
\label{LAMJAY} \lambda_J(x) := \frac{1}{3}\lambda\big(N_J(x)\big) &&(x \in J)
\end{align}
and satisfying the following conditions, for all $x,y \in J$.
\begin{align*}
\lambda_J(x) = \infty \Longleftrightarrow\,\,& x = 0, \\
\lambda_J(x + y)\geq\,\,&\min \{\lambda_J(x),\lambda_J(y)\}, \\
\lambda_J(U_xy) =\,\,&2\lambda_J(x) + \lambda_J(y).
\end{align*}
In particular, 
\[
\mfo_J := \{x \in J \mid \lambda_J(x) \geq 0\}
\]
is an $\mfo$-subalgebra of $J$, called the \emph{valuation algebra} of $\lambda_J$ (or of $J$), containing
\[
\mfp_J := \{x \in J \mid \lambda_J(x) > 0\}
\]
as its unique maximal ideal, called the \emph{valuation ideal} of $\lambda_J$. The invertible elements of $\mfo_J$ may be described as
\[
\mfo_J^\times = \{x \in J\mid \lambda_J(x) = 0\} = \mfo_J \setminus \mfp_J
\]
and, therefore, $\bar J := \mfo_J/\mfp_J$ is a Jordan division algebra over $\bar F$. We call $\bar J$ the \emph{residue algebra} of $J$. The natural map from $\mfo_J$ to $\bar J$ will again be denoted by $x \mapsto \bar x$. We also note that $\lambda_J$ preserves arbitrary powers and 
\begin{align}
\label{LAMSHA} \lambda_J(x^\sharp) = 2\lambda_J(x)
\end{align}
for all $x \in J$ since we may assume $x \ne 0$, so $\lambda_J(x^\sharp) = \lambda_J(N_J(x)x^{-1}) = \lambda(N_J(x)) - \lambda_J(x) = 3\lambda_J(x) - \lambda_J(x)=2\lambda_J(x)$.

\subsection{Examples} \label{ss.CUASDI} Let $A$ be a finite-dimensional cubic associative division algebra over $F$. The preceding considerations apply to $J := A^{(+)}$ and show, in obvious notation, $\lambda_J = \lambda_A$, $\mfo_J= \mfo_A^{(+)}$, $\mfp_J= \mfp_A$, and $\bar J = \bar A^{(+)}$.

\subsection{Ramification} \label{ss.RAMIF} Returning to the cubic Jordan division $F$-algebra $J$ of \ref{ss.CUJOCO}, we deduce from \cite[Korollar of Lemma~3.1]{MR0347922} that $\lambda_J(J^\times)$ is an additive subgroup of $\IQ$ which in turn contains $\IZ$ as a subgroup of finite index \cite[5.3]{MR0347922}. This index is called the \emph{ramification index} of $J$ over $F$, denoted by $e_{J/F}$. By \cite[Kor.~1 of Prop.~6.4]{MR0347922}, the ramification index divides the degree of $J$ and hence is either $1$ or $3$; in particular, it is odd, so we may conclude from \cite[Satz~6.3(c)]{MR0347922} that the residue degree of $J$ over $F$, denoted by $f_{J/F}$ and defined as in \cite[6.3]{MR0347922}, agrees with the dimension of $\bar J$ over $\bar F$. Moreover, by \cite[Satz~6.3(d)]{MR0347922} it satisfies the \emph{fundamental equality}
\begin{align}
\label{FUNDEQ} e_{J/F}f_{J/F} = \dim_F(J).
\end{align}
We clearly have $\mfp\mfo_J \subseteq \mfp_J$, and since $\mfo_J$ is a free module of rank $\dim_F(J)$ over the PID $\mfo$ \cite[Satz~6.3(a)]{MR0347922}, one checks  that
\begin{align}
\label{RAMONE} e_{J/F} = 1 \Longleftrightarrow \lambda_J(J^\times) = \IZ \Longleftrightarrow \mfp_J= \mfp\mfo_J.
\end{align}
$J$ is said to be \emph{unramified} (resp., \emph{ramified}) if $e_{J/F} = 1$ (resp., $e_{J/F} = 3$) and $\bar J$ is separable over $\bar F$.

\subsection{Extending the cubic structure} \label{ss.EXTCUB} Our first aim in this section will be to close a gap in the proof of \cite[Prop.~3]{MR0379620}\footnote{The gap reveals itself in the sloppy definition of a cubic form.} by deriving a canonical cubic structure on $\bar J$ over $\bar F$ from the given one on $J$ over $F$. We proceed in two steps, the first one passing from $J$ over $F$ to $\mfo_J$ over $\mfo$, the second from $\mfo_J$ over $\mfo$ to $\bar J$ over $\bar F$. The following lemma paves the way for the first step.

\begin{lem} \label{l.VANPOL}
Let $R$ be a commutative ring, $M,N,N^\prime$ be $R$-modules, $f\:M \to N$ a polynomial law over $R$ and $i\:N \to N^\prime$ an injective $R$-linear map. If $i \circ f = 0$ as a polynomial law over $R$, the $f = 0$ as a polynomial law over $R$.
\end{lem}

\begin{proof}
Let $\bfT = (\bft_i)_{i \in \IN}$ be a family of indeterminates. The corresponding polynomial ring $S := R[\bfT]$ is free, hence flat, as an $R$-module, so the $S$-linear extension $i_S\:N_S \to N^\prime_S$ continues to be injective. Hence $i_S \circ f_S = 0$ implies $f_S = 0$ as a set map $M_S \to N_S$, and the lemma follows from \cite[Cor.~12.11]{MR4806163}.	
\end{proof}

\begin{prop} \label{p.JAYOJA}
Let $J$ be a cubic Jordan division algebra over $F$. Writing $i\:\mfo \hookrightarrow F$ for the inclusion, there is a unique structure  of a cubic Jordan $\mfo$-algebra on $\mfo_J$ such that the inclusion $i_J\:\mfo_J \hookrightarrow J$ is an $i$-semi-linear homomorphism of cubic Jordan algebras.
\end{prop}

\begin{proof}
By \cite[34.10(b)]{MR4806163}, the assertion amounts to the following: there exists a unique cubic form $N_{\mfo_J}\:\mfo_J \to \mfo$ making $\mfo_J$ a cubic Jordan $\mfo$-algebra and rendering
\begin{align}
\vcenter{\label{ISEMILI}\xymatrix{\mfo_J \ar@{^{(}->}[r] _{i_J} \ar[d]_{N_{\mfo_J}} & J \ar[d]^{N_J} \\
\mfo \ar@{^{(}->}[r]_i & F}} 
\end{align}
a commutative $i$-semi-linear polynomial square in the sense that $_{\mfo}i \circ N_{\mfo_J} = \,_{\mfo}i_J \circ\,_{\mfo}N_J$ as polynomial laws over $\mfo$; here the left-hand index ``$\mfo$''  refers to the restriction of scalars  for polynomial laws \cite[12.27]{MR4806163}. Hence uniqueness follows from Lemma~\ref{l.VANPOL}.

In order to prove existence, we begin by showing
\begin{align}
\label{SHATESO} \mfo_J^\sharp \subseteq \mfo_J, \quad N_J(\mfo_J) + T_J(\mfo_J) + T_J(\mfo_J,\mfo_J) + S_J(\mfo_J) \subseteq \mfo.
\end{align}
While the first relation follows from \eqref{LAMSHA}, we have $N_J(\mfo_J) \subseteq \mfo$ by \eqref{LAMJAY}, and the proof of \cite[Satz~5.1]{MR0347922} yields $T_J(\mfo_J) + S_J(\mfo_J) \subseteq \mfo$, hence $S_J(\mfo_J,\mfo_J) \subseteq \mfo$ after linearization. Now \cite[(33a.13)]{MR4806163} also implies $T_J(\mfo_J,\mfo_J) \subseteq \mfo$ and completes the proof of \eqref{SHATESO}.

We have noted before that $\mfo_J$ is a free $\mfo$-module of rank $n := \dim_F(J)$. Let $(e_1,\dots,e_n)$ be an $\mfo$-basis of $\mfo_J$, hence an $F$-basis of $J$. For $R \in \oalg$, define a set map 
\[
(N_{\mfo_J})_R\:(\mfo_J)_R \longrightarrow R 
\]
by
\begin{align*}
(N_{\mfo_J})_R(x) := \sum_{i=1}^n N_J(e_i)r_i^3 + \sum_{i\ne j}T_J(e_i^\sharp,e_j)r_i^2r_j + \sum_{i<j<l}T_J(e_i \times e_j,e_l)r_ir_jr_l
\end{align*}
for all $r_1,\dots,r_n \in R$ and $x := \sum_{i=1}^n e_i \otimes_{\mfo} r_i \in (\mfo_J)_R$. By \eqref{SHATESO}, the right-hand side of the displayed equation belongs to $R$, and one checks that the set maps $(N_{\mfo_J})_R$ vary functorially with $R \in \oalg$, hence define a cubic form $N_{\mfo_J}\:\mfo_J \to \mfo$. There is a natural identification of $\mfo_J \otimes_{\mfo} F$ with $J$ as Jordan $F$-algebras matching $x \otimes_{\mfo} \alpha$ for $x \in \mfo_J$ and $\alpha \in F$ with $\alpha x \in J$. One checks that, under this identification, $N_{\mfo_J} \otimes_{\mfo} F = N_J$, and \eqref{ISEMILI} follows from \cite[(12.29.2)]{MR4806163}. By \eqref{SHATESO}, restricting the adjoint of $J$ to $\mfo_J$ gives a quadratic map $\mfo_J \to \mfo_J$, $x \mapsto x^\sharp$, called the adjoint of $\mfo_J$, and it suffices to show that $\mfo_J$ together with the base point $1_{\mfo_J}$, its adjoint, and the norm $N_{\mfo_J}$ is a cubic norm structure over $\mfo$. It is certainly a cubic array by \cite[33.1]{MR4806163} since $\mfo_J$ is a free $\mfo$-module. Moreover, the identities of \cite[Exc.~33.14]{MR4806163}, being valid in all of $J$, in particular hold in $\mfo_J$, and the assertion follows.
\end{proof}

\begin{cor} \label{c.NOSTRO}
We have $(\mfo_J)_F = J$ as cubic Jordan $F$-algebras, and adjoint, (bi-)linear trace, quadratic trace of the cubic Jordan $\mfo$-algebra $\mfo_J$ are obtained by restricting the corresponding objects for $J$ to $\mfo_J$.	
\end{cor}

\begin{proof}
This follows either by consulting the preceding proof or by combining the results of \cite[12.29 and 34.10]{MR4806163}.	
\end{proof}

For our second step, we require a lemma that is surely well known but seems to lack a convenient reference.

\begin{lem} \label{l.FISETR}
Let $E/F$ be a finite separable field extension. Then $T_{E/F}(\mfp_E) \subseteq \mfp$.	
\end{lem}

\begin{proof}
Write $\sigma_i\:E \to E^\prime$, $1\le i\le n := [E:F]$, for the $n$ distinct $F$-embeddings of $E$ into its Galois closure $E^\prime/F$. For each $i$, the $F$-embedding $\sigma_i$ extends to some element $\sigma_i^\prime \in \Gal(E^\prime/F)$, and we have $\lambda_{E^\prime} \circ \sigma_i^\prime = \lambda_{E^\prime}$. For $u \in \mfp_E$ we therefore conclude $\lambda_{E^\prime}(\sigma_i(u)) = \lambda_{E^\prime}(\sigma_i^\prime(u)) = \lambda_{E^\prime}(u) = \lambda_E(u) > 0$, hence $\sigma_i(u) \in \mfp_{E^\prime}$. Thus $T_{E/F}(u) = \sum_{i=1}^n \sigma_i(u) \in \mfp_{E^\prime} \cap F = \mfp$.
\end{proof}	

\begin{prop} \label{p.FRECUI}
Let $J$ be a cubic Jordan division algebra over $F$. Then $(\mfp,\mfp_J)$ is a cubic ideal of $\mfo_J$ in the sense of \cite[Exc.~34.21]{MR4806163}. In particular, writing $\varrho\:\mfo \to \bar F$, $\varrho_J\:\mfo_J \to \bar J$ for the canonical projections, there is a unique cubic Jordan algebra structure on $\bar J$ making $\varrho_J$ a $\varrho$-semi-linear homomorphism of cubic Jordan algebras in the sense of \cite[34.10(b)]{MR4806163}.
\end{prop}

\begin{proof}
According to \cite[Exc.~34.21]{MR4806163}, it suffices to prove the following:
\begin{itemize}
\item [(i)] $\mfp\mfo_J \subseteq \mfo_J$.

\item [(ii)] $T_J(x,y), T_J(x^\sharp,y), N_J(x) \in \mfp$ for all $x \in \mfp_J$, $y \in \mfo_J$.

\item [(iii)] There exists an $\mfo$-linear map $\vartheta\:\mfo_J \to \bar F$ satisfying $\vartheta(1_J) = 1_{\bar F}$ and $\vartheta(\mfp_J) = \{0\}$.
\end{itemize}
(i) is obvious. In the first relation of (ii), applying Cor.~\ref{c.NOSTRO}, we may assume $y \in \mfo_J^\times$. Passing to the $y$-isotope of $\mfo_J$ and invoking \cite[(33.11.5)]{MR4806163}, we may even assume $y = 1_{\mfo_J}$, $x \notin F1_J$ and $T_J(x) \ne 0$. Then $E := F[x] \subseteq J$ is a separable cubic subfield \cite[Prop.~46.6]{MR4806163}, and Lemma~\ref{l.FISETR} implies $T_J(x,1_{\mfo_J})= T_{E/F}(x) \in \mfp$, yielding the first relation of (ii). The second and third now follow immediately from  the relations of \ref{ss.CUJOCO}. In (iii), since $1_{\bar J} \in \bar J$ is trivially unimodular, there exists a linear form $\vartheta^\prime\:\bar J \to \bar F$ satisfying $\vartheta^\prime(1_{\bar J}) = 1_{\bar F}$, and $\vartheta := \vartheta^\prime \circ \varrho_J$ does the job.
\end{proof}

Here is a first application of the preceding result.

\begin{prop} \label{p.NILOJA}
Let $J$ be a cubic Jordan division algebra over $F$. With the natural identification $(\mfo_J)_{\bar F} = \mfo_J/\mfp\mfo_J$, we have	
\[
\Nil\big((\mfo_J)_{\bar F}\big) =\mfp_J/\mfp\mfo_J = \{x \in (\mfo_J)_{\bar F}) \mid x^3 = 0\} = \{x \in (\mfo_J)_{\bar F} \mid \text{$x$ is nilpotent}\}. 
\]
\end{prop}

\begin{proof}
We show cyclically that each link of the preceding chain is contained in the next. To begin with,  since $\mfp_J \subseteq \mfo_J$ is the unique maximal ideal, so is $\mfp_J/\mfp\mfo_J \subseteq (\mfo_J)_{\bar F}$. Hence $\Nil((\mfo_J)_{\bar F}) \subseteq \mfp_J/\mfp\mfo_J$. Next, let $\varrho\:\mfo_J \to (\mfo_J)_{\bar F}$ be the natural map and $0 \ne u \in \mfp_J$. Then $u^3 = N_J(u)(N_J(u)^{-1}u^3)$, where the first factor belongs to $\mfp$ and the second one to $\mfo_J^\times$. Hence $u^3 \in \mfp\mfo_J$., so $\varrho(u)^3 = 0$. And finally, suppose $u \in \mfo_J$ has $\varrho(u) \in (\mfo_J)_{\bar F}$ nilpotent. If $u \notin \mfp_J$, the $u \in \mfo_J^\times$, forcing $\varrho(u) \in (\mfo_J)_{\bar F}^\times$, a contradiction. Hence $u \in \mfp_J$, and we conclude that the set of nilpotents in $(\mfo_J)_{\bar F}$, being the same as $\mfp_J/\mfp\mfo_J$, is an ideal in $(\mfo_J)_{\bar F}$ and thus belongs to $\Nil((\mfo_J)_{\bar F})$.
\end{proof}
	
\begin{prop} \label{p.CHAUNR}
For a Freudenthal division $F$-algebra $J$ of dimension $> 1$, the following conditions are equivalent.
\begin{itemize}
\item [(i)] $J$ is unramified.

\item [(ii)] $\mfo_J$ is a regular Freudenthal algebra over $\mfo$.

\item [(iii)] $\mfo_J$ is a Freudenthal algebra over $\mfo$.

\item [(iv)] $\mfo_J$ is a regular cubic Jordan algebra over $\mfo$.

\item [(v)] There exists a regular cubic Jordan $\mfo$-algebra $\mfo_J^\prime \subseteq J$ such that the inclusion $\mfo_J^\prime \hookrightarrow J$ is an ($\mfo \hookrightarrow F$)-semi-linear embedding and canonically induces an identification $\mfo_J^\prime \otimes_{\mfo} F = J$ as cubic Jordan $F$-algebras.
\end{itemize}
In this case, $\mfo_J^\prime = \mfo_J$, for any cubic Jordan $\mfo$-algebra $\mfo_J^\prime$ satisfying the hypotheses of ($\mathrm{v}$). 
\end{prop}

\begin{proof}
We will establish the following chain of implications:
\[
(\mathrm{i}) \Rightarrow (\mathrm{ii}) \Rightarrow (\mathrm{iii}) \Rightarrow (\mathrm{ii}) \Rightarrow (\mathrm{iv}) \Rightarrow (\mathrm{v}) \Rightarrow (\mathrm{iv}) \Rightarrow (\mathrm{i}) .
\] 

(i)$\Rightarrow$(ii). If $J$ is unramified, then \eqref{RAMONE} shows $\mfp_J= \mfp\mfo_J$, and $(\mfo_J)_{\bar F} = \bar J$ is a regular Freudenthal division algebra \cite[Thm.~46.8]{MR4806163}. Letting $(e_i)$ be a basis of $\mfo_J$ as an $\mfo$-module, the determinant of $(T_{\mfo_J}(e_i,e_j))$ becomes non-zero (in $\bar F$) after reduction $\bmod~\mfp$, hence is a unit in $\mfo$, forcing $\mfo_J$ to be a regular cubic Jordan $\mfo$-algebra. Given any field $K \in \oalg$, the unit homomorphism $\mfo \to K$ factors through $F$ (resp., $\bar F$), so Cor.~\ref{c.NOSTRO} (resp., Prop.~\ref{p.FRECUI}) yields $(\mfo_J)_K = J_K$ (resp., $(\mfo_J)_K =\bar J_K$), and this is either cubic \'etale or simple as a cubic Jordan $K$-algebra since $J$ (resp., $\bar J$) is Freudenthal. By definition \cite[39.8]{MR4806163}, therefore, $\mfo_J$ is a regular Freudenthal algebra over $\mfo$. 

(ii) $\Rightarrow$ (iii). Obvious.

(iii) $\Rightarrow$ (ii). Since $J$ is not of rank $6$ over $F$ \cite[Thm.~46.8]{MR4806163}, neither is $\mfo_J$ over $\mfo$. Hence the assertion follows from \cite[Cor.~39.15]{MR4806163}.

(ii) $\Rightarrow$ (iv) $\Rightarrow$ (v). Obvious.

(v) $\Rightarrow$ (iv). The set map $N_J\:J \to F$ restricts to the set map $N_{\mfo_J^\prime}\:\mfo_J^\prime \to \mfo$. Hence $\mfo_J^\prime \subseteq \mfo$ is a cubic Jordan $\mfo$-subalgebra. By regularity, therefore, $\mfo_J= \mfo_J^\prime \oplus \mfo_J^{\prime\perp}$ \cite[Lemma~11.10]{MR4806163}. But both $\mfo_J$ and $\mfo_J^\prime$ are free $\mfo$-modules of rank $\dim_F(J)$, and (iv) holds.	

(iv) $\Rightarrow$ (i). If $\mfo_J$ is a regular cubic Jordan $\mfo$-algebra, then $(\mfo_J)_{\bar F}$ is a regular cubic Jordan $\bar F$-algebra, so its nil radical, which agrees with $\mfp_J/\mfp\mfo_J$ by Prop.~\ref{p.NILOJA}, must be zero \cite[Exc.~34.23]{MR4806163}. Hence $e_{J/F} = 1$ by \eqref{RAMONE}, and $\bar J = (\mfo_J)_{\bar F}$ is a regular cubic Jordan division algebra. Thus $J$ is unramified. Moreover, the final statement of the proposition also holds, thanks to the proof of the implication (v)$\Rightarrow$(iv).
\end{proof}

\begin{rem} \label{r.UNRET}
\emph{The argument in the proof of the implication (i)$\Rightarrow$(ii) in Prop.~\ref{p.CHAUNR} to show that $\mfo_J$ is a regular cubic $\mfo$-algebra can be repeated verbatim to show that} a quadratic field extension $K/F$ is unramified if and only if $\mfo_K$ is a quadratic \'etale $\mfo$-algebra.    
\end{rem}

\begin{eg} \label{e.UNRAAS}
Let $A$ be an unramified cubic associative division $F$-algebra. $\vssp$ \\ 
(a) We deduce from Pop.~\ref{p.CHAUNR} (applied to $A^{(+)}$) that $\mfo_A$ is a regular cubic associative $\mfo$-algebra. For $\mu \in \mfo^\times \setminus N_A(A^\times)$, therefore, the first Tits construction $J := J(A,\mu)$ is a regular Freudenthal division algebra over $F$ \cite[Cors.~42.14 and 46.12]{MR4806163} which, since first Tits constructions are stable under base change, contains $\mfo_J^\prime := J(\mfo_A,\mu)$ in such a way that the hypotheses of Prop.~\ref{p.CHAUNR}~(v) are fulfilled. Thus $J$ is unramified, $\mfo_J^\prime = \mfo_J$, and \eqref{RAMONE} implies $\bar J = \mfo_J/\mfp\mfo_J = \mfo_J \otimes_\mfo \bar F = J(\bar A,\bar \mu)$ over $\bar F$. $\vssp$ \\
(b) On the other hand, if $\pi$ is a prime element of $\mfo$, then we claim that \emph{the first Tits construction $J := J(A,\pi) = A \oplus Aj_1 \oplus Aj_2$ is a ramified Freudenthal division $F$-algebra}. Indeed, since $A$ is unramified, we have $\lambda(N_A(x)) \equiv 0 \bmod 3$ for all $x \in A^\times$, while $\lambda(\pi) = 1$. Thus $\pi \notin N_A(A^\times)$, so $J$ is a Freudenthal division $F$-algebra, which must be ramified since $\lambda_J(j_1) = \frac{1}{3}$. By the fundamental equality \eqref{FUNDEQ}, $\bar J$ and $\bar A^{(+)} \subseteq \bar J$ have the same $\bar F$-dimension, and we conclude $\bar J = \bar A^{(+)}$. We will see in Cor.~\ref{c.TRICHO} below that, conversely, every ramified Freudenthal division algebra over $F$ has the form described above.  	
\end{eg}

\begin{lem} \label{l.UNRQUA}
Let $K/F$ be a quadratic field extension and $J$ a cubic Jordan division algebra over $F$ having $e_{J/F} = 1$. Then $J_K$ is a cubic Jordan division algebra over $K$ and $e_{J_K/K} = 1$. More precisely, there is a natural identification $\overline{J_K} = \bar J_{\bar K}$ as cubic Jordan $\bar K$-algebras.	
\end{lem}

\begin{proof}
$J_K$ is a cubic Jordan division algebra over $K$ \cite[Cor.~46.3]{MR4806163} containing $J$ as a cubic Jordan division $F$-subalgebra such that $\lambda_{J_K}\vert_J = \lambda_J$. Hence there is a natural $\bar F$-homomorphism$\, \bar J \to \overline{J_K}$, which in turn determines a canonical $\bar K$-homomorphism
\[
\bar J_{\bar K} \longrightarrow \overline{J_K}
\]	
of cubic Jordan $\bar K$-algebras \cite[Exc.~34.22~(a)]{MR4806163}. But $\bar J_{\bar K}$ is a cubic Jordan division algebra, so this homomorphism is injective. On the other hand, since $e_{J/F} = 1$, and by \cite[5.3]{MR0347922}, we have
\begin{align*}
\dim_F(J) = \dim_{\bar F}(\bar J) = \dim_{\bar K}(\bar J_{\bar K}) \le \dim_{\bar K}(\overline{J_K}) \le \dim_K(J_K) = \dim_F(J),
\end{align*}
hence erquality everywhere, and the lemma follows.
\end{proof}


\section{Classification theorems} \label{s.CLATHE} Using the preparations assembled in the preceding section, we will now be able to tackle the problem of classifying Freudenthal division algebras over $F$, more precisely, of reducing the classification problem to the corresponding problem over the residue field $\bar F$. In the important special case of Albert division algebras, this reduction has been carried out already in \cite{MR0379620}.The approach adopted here will allow for a uniform treatment of all Freudenthal division algebras over $F$. Special attention will be paid to the question of which Freudenthal division $F$-algebras are homogeneous.

The main ingredients of the reduction procedure needed for our results may be described as follows.

\begin{rem} \label{r.UNRAFI}
\emph{Let $\msB :=(K,B,\tau)$ be an associative involutorial system over $F$ as in \ref{ss.SETICO} such that $K$ is a field. Unless $\bar K/\bar F$ is an inseparable quadratic field extension, $\bar\msB := (\bar K,\bar B,\bar\tau)$, where $\bar\tau$ stands for the $\bar K/\bar F$-involution canonically induced by $\tau$ on $\bar B$, is an associative involutorial system over $\bar F$. Moreover, if $(p,\mu)$ is an admissible scalar for $\msB$, then without changing $\msB$, there is no harm in assuming $N_B(p) = n_K(\mu) = 1$ (\cite[(39.2)(2)]{MR1632779}), \cite[Exc.~44.33(c)]{MR4806163}), and then ($\bar p,\bar\mu$) is an admissible scalar for $\bar\msB$.}

\emph{In our subsequent considerations, in order to avoid confusion with residue homomorphisms, we deviate from the notation of \cite{MR4806163} by indicating the conjugation of a conic algebra \emph{not} by $x \mapsto \bar x$ \emph{but} by $x \mapsto x^\ast$.}
\end{rem}
%
\begin{thm} \label{UNSETI}
Let $\msB= (K,B,\tau)$ be an associative involutorial system over $F$ and $(p,\mu)$ an admissible scalar for $\msB$. Assume that the second Tits construction $J := J(\msB,p,\mu)$ is a division algebra, that the $F$-subalgebra $J_0 := H(\msB) \subseteq J$ is unramified, and that $K$ is a field. Then
\begin{itemize}
\item [(a)] $B$ is an unramified cubic associative division algebra over $K$.

\item [(b)] $K/F$ is an unramified quadratic field extension.

\item [(c)] $\mfo_{\msB} := (\mfo_K,\mfo_B,\tau\vert_{\mfo_B})$ is an associative involutorial system over $\mfo$, and if we assume $N_B(p) = n_K(\mu) = 1$ as in Remark~\ref{r.UNRAFI}, then
\begin{itemize}
\item [(c1)] ($p,\mu)$ is an admissible scalar for $\mfo_{\msB}$ such that the corresponding second Tits construction satisfies
\begin{align}
\label{MFOJAY} \mfo_J = J(\mfo_{\msB},p,\mu) = \mfo_{J_0} \oplus \mfo_Bj 
\end{align}
as regular Freudenthal $\mfo$-algebras.

\item [(c2)] $\bar\msB:= (\bar K,\bar B,\bar\tau)$ is an associative involutorial system over $\bar F$ allowing $(\bar p,\bar\mu)$ as an admissible scalar, $J$ is unramified, and there is a natural isomorphism $\bar J\cong J(\bar\msB,\bar p,\bar\mu)$ as cubic Jordan $\bar F$-algebras. 
\end{itemize}
\end{itemize}
\end{thm}

\begin{proof}
(a) $K/F$ is a separable quadratic field extension by \cite[Cor.~46.11]{MR4806163}. From \cite[Exc.~44.29(a)]{MR4806163} we therefore deduce that there is a natural identification $J_{0K} =B^{(+)}$ as cubic Jordan $K$-algebras, so Lemma~\ref{l.UNRQUA} implies $e_{B/K} = 1$ and $\bar B^{(+)} \cong \bar J_{0\bar K}$. Since $\bar J_0$ is separable over $\bar F$ by hypothesis, so is $\bar B$ over $\bar K$, and we have shown (a).

(b) Arguing indirectly, assume that $K$ is \emph{not} unramified over $F$. Then $\bar K = \bar F$ or $\bar K/\bar F$ is an inseparable quadratic field extension. In any event, the conjugation of $\bar K$ as a conic $\bar F$-algebra is the identity. We may assume $N_B(p) = 1 = n_K(\mu)$ as in (c) and conclude $1 = \mu\mu^\ast$, hence $\bar 1 = \bar\mu\bar\mu^\ast = \bar\mu^2$, i.e., $\bar\mu = \pm 1$. Thus $\mu \equiv\pm 1 \bmod\mfp_K$, and \cite[Cor. of Lemma~8]{MR0379620} shows $\mu = N_B(w)$ for some $w \in B^\times$, a contradiction to \cite[Thm.~46.10]{MR4806163} and $J$ being a division algebra. This completes the proof of (b).

(c1) Since $\mfo_K$ is quadratic \'etale over $\mfo$ by (b) and Remark~\ref{r.UNRET}, it follows that $\mfo_{\msB} := (\mfo_K,\mfo_B,\tau\vert_{\mfo_B})$ is an associative involutorial system over $\mfo$ allowing $(p,\mu)$ as an admissible scalar. Hence $\mfo_J^\prime := J(\mfo_{\msB},p,\mu)$ makes sense as a cubic Jordan $\mfo$-algebra, which is regular by \cite[Cor.~44.22]{MR4806163}. The associative valuation $\lambda_B$ of $B$ restricts to a Jordan valuation on $J_0$ extending $\lambda$ and thus agrees with $\lambda_{J_0}$ \cite[Satz~5.1]{MR0347922}. Thus $H(\mfo_{\msB}) = \mfo_{J_0}$, and the second equation of \eqref{MFOJAY} holds. Moreover, since the second Tits construction displayed in \eqref{MFOJAY} is stable under base change, we have $(\mfo_J^\prime)_F = J$, so the hypotheses of Prop.~\ref{p.CHAUNR}~(v) are fulfilled, and we conclude $\mfo_J^\prime = \mfo_J$. This completes the proof of (c1).

(c2) The first assertion about $\bar\msB$, $\bar p$, $\bar\mu$ being obvious, we combine (c1) with Prop.~\ref{p.CHAUNR} to conclude that $J$ is unramified.  We clearly have $\bar J_0 \subseteq H(\bar B,\bar\tau)$, and combining \cite[Exc.~44.29]{MR4806163} with Lemma~\ref{l.UNRQUA}, we obtain
\[
H(\bar B,\bar\tau)_{\bar K} \cong \bar B^{(+)} \cong \overline{(J_0)_K} \cong (\bar J_0)_{\bar K}. 
\]
Comparing $\bar F$-dimensions, this implies $\bar J_0 = H(\bar\msB)$. Now \eqref{MFOJAY} yields $\mfp\mfo_J= \mfp\mfo_{J_0} \oplus \mfp\mfo_Bj$, and we conclude $\bar J \cong \bar J_0 \oplus \bar Bj = H(\bar\msB) \oplus \bar Bj =J(\bar B,\bar p,\bar\mu)$, as claimed.
\end{proof}
	
\begin{cor} \label{c.TRICHO}
For any cubic Jordan division algebra $J$ over $F$, precisely one of the following conditions holds.
\begin{itemize}
\item [(a)] $\bar J \cong E^{\prime(+)}$, for some purely inseparable field extension $E^\prime/\bar F$ of characteristic $3$ and exponent at most $1$.

\item [(b)] $J$ is unramified.

\item[(c)] $J$ is a ramified first Tits construction.
\end{itemize}
Moreover, in case (c), there exist an unramified cubic associative division algebra $A$ over $F$ and a prime element $\pi \in \mfo$ such that $J \cong J(A,\pi)$. 
\end{cor}

\begin{proof}
We may assume that $J$ is a ramified Freudenthal division $F$-algebra, of dimension $3^n$, $1 \leq n \leq 3$, and must show that the final statement of the corollary holds. By definition, $\bar J$ is a Freudenthal division algebra over $\bar F$, of dimension $3^{n-1}$. Note that $\bar J$ has $n - 1$ generators as a cubic Jordan algebra, which is clear for $n = 1,2$ and follows from \cite[Cor. of Lemma~9]{MR0379620} for $n = 3$. Lifting these generators from $\bar J$ to $\mfo_J$, the resulting quantities by \cite[Exc.~33.15]{MR4806163} and the fundamental equality \eqref{FUNDEQ}, generate an unramified Freudenthal division $F$-algebra $J_0 \subseteq J$ having dimension $3^{n-1}$ and residue algebra $\bar J$. Thus $(J,J_0)$ is a Freudenthal pair over $F$ in the sense of \cite[45.1]{MR4806163}, hence, by \cite[Thm.~45.10]{MR4806163}, admits \'etale elements. Combining \cite[Cor.~44.17]{MR4806163} with \cite[Lemma~A4.3]{garibaldi2025solutionsexercisesbookalbert}, we find an associative involutorial system $\msB= (K,B,\tau)$ over $F$ as well as an admissible scalar $(p,\mu)$ for $\msB$ such that $J \cong J(\msB,p,\mu)$ under an isomorphism matching $J_0$ with $H(\msB)$. Here $K$ is quadratic \'etale over $F$ \cite[Cor.~46.11]{MR4806163} and, thanks to Thm.~\ref{UNSETI}~(c2), cannot be a field because, otherwise, $J$ would be unramified. Summing up, therefore, $K =F \times F$ is split quadratic \'etale over $F$, $B =A \times A^{\mathrm{op}}$ for some cubic associative $F$-algebra $A$, and $J = J(A,\kappa)$ for some $\kappa \in F^\times$ is a first Tits construction \cite[Cor.~44.21]{MR4806163}. Since $J_0 = H(\msB) \cong A^{(+)}$, we conclude that $A$ is unramified. We may multiply $\kappa$ by an appropriate invertible norm of $A$ \cite[Exc.~42.22~(a)]{MR4806163} to assume $0 \le \lambda(\kappa) \le 2$, where we can rule out $\lambda(\kappa) = 0$, i.e., $\kappa \in \mfo^\times$, since this would imply by Example~\ref{e.UNRAAS}~(a) that $J$ is unramified. Thus $\lambda(\kappa) \in \{1,2\}$. Replacing $A$ by $A^{\mathrm{op}}$ if necessary, we may assume $\lambda(\kappa) = 1$ \cite[Exc.~42.22~(c), (d)]{MR4806163}, i.e., $\pi := \kappa$ is a prime element of $\mfo$.
\end{proof}

\begin{cor} \label{c.RAMISP}
If $J$ is a ramified Freudenthal division $F$-algebra, then so is every isotope of $J$.	
\end{cor}

\begin{proof}
For any $p \in J^\times$, we deduce from \cite[Satz~6.3~(b)]{MR0347922} that $J^{(p)}$ has ramification index $3$, so it suffices to show that its residue algebra is separable. In order to see this, we apply Cor.~\ref{c.TRICHO} and realize
\[
 J = J(A,\pi) = A \oplus Aj_1 \oplus Aj_2
\]	
as a first Tits construction, where $A$ is an unramified cubic associative division algebra and $\pi \in \mfo$ is a prime element. We may clearly assume $\lambda_J(p) \in \{0,1,2\}$, where $\lambda_J(p) = 0$ implies $p \in \mfo_J^\times$, and $\overline{J^{(p)}} \cong \bar J^{(\bar p)} \cong \bar A^{(+)(\bar p)} \cong \bar A^{(+)}$ is separable. Hence, passing to $A^{\mathrm{op}}$ if necessary, we may assume $\lambda_J(p) = 1$ \cite[Exc.~42.22~(c), (d)]{MR4806163}. From \eqref{UOPCUB}, \eqref{FITBLT} and \eqref{SETBLT} we deduce $U_{j_2}1_J = \pi j_1$. By \cite[Cor.~31.23]{MR4806163}, therefore, $J^{(j_1)} \cong J$. On the other hand, setting $q := U_{j_1^{-1}p}$, \cite[Thm.~31.10]{MR4806163} implies we have $J^{(p)} \cong (J^{(j_1)})^q$ and $q \in \mfo_{J^{(j_1)}}^\times$. Thus we are reduced to the case$\lambda_J(p) = 0$, which has been settled before.
\end{proof}

\begin{cor} \label{RAFRHO} 
Ramified Freudenthal division algebras over $F$ are homogeneous.
\end{cor}

\begin{proof}
Let $J$ be a ramified Freudenthal division algebra over $F$, of dimension $3^n$, $0 \leq n \leq3$. For $n = 0,1$, there is nothing to prove. If $n = 3$, $J$ is an Albert algebra, and the result holds under far more general circumstances (Thm.~\ref{t.FITIHO}) since $J$ by Cor.~8.4 is a first Tits construction. We are this left with the case that $J$ has dimension $9$. Writing $J = H(B,\tau)$ for some central simple associative $F$-algebra $(B,\tau)$ of degree $3$ with unitary involution, and setting $K := \Cent(B)$, we must show by Prop.~\ref{p.HOMDIS} that all $K/F$- involutions of $B$ are distinguished. Let $\tau^\prime$ be any $K/F$-involution of $B$. By \ref{ss.ISDINV}, we have $\tau^\prime = \tau^p$ for some $p \in J^\times$, hence $H(B,\tau^\prime) \cong J^{(p)}$. But $J^{(p)}$ is ramified by Cor.~\ref{c.RAMISP}, hence a first Tits construction by Cor.~\ref{c.TRICHO}. By Thm.~\ref{t.CHADIS}, therefore, $\tau^\prime$ is distinguished.
\end{proof}	

It remains to investigate \emph{unramified} Freudenthal division $F$-algebras of dimension $9$, the case of dimension $1$ (i.e., the base field) being trivial, of dimension $3$ (i.e., separable cubic field extensions) being part of ordinary valuation theory, and of dimension $27$ (i.e., Albert division algebras) having been settled in \cite{MR0379620}. As before, the question of homogeneity will be of central importance in our investigation. We begin by setting the terminology straight.

\begin{ter} \label{t.CEASDI}
\emph{Let $k$ be an arbitrary field. In analogy to \cite[\S4]{MR0379620}, we define a central associative division algebra of degree $3$ with unitary involution over $k$ as a pair $(B,\tau)$, where $B$ is an associative division $k$-algebra having degree $3$ over its center $K := \Cent(B)$, a separable quadratic extension field of $k$, and $\tau$ is a $K/k$-involution of $B$. Specializing $k$ to our complete field $F$, we say that $(B,\tau)$ is \emph{unramified} if $B$ is unramified over $F$. In this case, $\IZ= \lambda(F^\times) \subseteq \lambda_K(K^\times) \subseteq \lambda_B(B^\times)$, and $B$ being unramified over $F$ implies that the extension $K/F$ has ramification index $1$. Moreover, by \cite[Satz~5.2]{MR0347922}, $\bar B$ is a central associative division algebra of degree $3$ over $\bar K$. Assuming that $\bar\tau$ induces the identity on $\bar K$ would imply that $\bar\tau$ itself is an involution of the first kind of $\bar B$, a contradiction (\cite[Chap.~8, Thm.~8.4]{MR770063}, \cite[Prop.~4.5.13]{GilleSz}). Thus $K/F$ is unramified, and $(\bar B,\bar\tau)$ is a central associative division algebra of degree $3$ with unitary involution over $\bar F$. It follows that $J := H(B,\tau)$ is an unramified Freudnthal division $F$-algebra of dimension $9$ satisfying $\bar J = H(\bar B,\bar\tau)$ \cite[Prop.~2]{MR0379620}}.
\end{ter}

\begin{thm} \label{t.UNRALI}
If $J$ is an unramified Freudenthal division algebra of dimension $9$ over $F$, then $\bar J$ is a Freudenthal division algebra of dimension $9$ over $\bar F$. Conversely, let $J^\prime$ be a Freudenthal division algebra over $\bar F$. Then there exists an unramified Freudenthal division algebra $J$ of dimension $9$ over $F$, unique up to isomorphism, such that $\bar J	\cong J^\prime$.
\end{thm}

\begin{proof}
The first part follows directly from the definitions. For the second part, write $J^\prime = H(B^\prime,\tau^\prime)$ for some central simple associative	algebra $(B^\prime,\tau^\prime)$ of degree $3$ with unitary involution over $\bar F$. We may assume that $K^\prime := \Cent(B^\prime)$ is a field, forcing $(B^\prime,\tau^\prime)$ to be a central associative division algebra of degree $3$ with unitary involution over $\bar F$. Applying \cite[Thm.~1]{MR0379620}, we find an unramified central associative division algebra $(B,\tau)$ of degree $3$ with unitary involution over $F$, unique up to isomorphism, such that $(\bar B,\bar\tau) \cong (B^\prime,\tau^\prime)$. By \ref{t.CEASDI}, therefore, $J := H(B,\tau)$ is an unramified Freudenthal division algebra of dimension $9$ over $F$ satisfying $\bar J \cong J^\prime$. Uniqueness of $J$ follows from uniqueness of $(B,\tau)$.
\end{proof}	

\begin{cor} \label{c.REDIIN}
If $(B,\tau)$ is an unramified associative division algebra of degree $3$ with unitary involution over $F$, then $\tau$ is distinguished if and only if $\bar\tau$ is distinguished.	
\end{cor}

\begin{proof}
$J := H(B,\tau)$ is an unramified Freudenthal division algebra of dimension $9$ for which \ref{t.CEASDI} implies $\bar J = H(\bar B,\bar\tau)$. Assume first $\bar\tau$ is distinguished. Then $\bar J$ is a first Tits construction (Thm.~\ref{t.CHADIS}), so $\bar J = J(E^\prime,\mu^\prime)$, for some separable cubic field extension $E^\prime/\bar F$ and some $\mu^\prime \in \bar F^\times \setminus N_{E^\prime}(E^{\prime\times})$. Lifting $E^\prime$ to an unramified cubic extension $E/F$ \cite[III, Thm.~2]{MR354618} and $\mu^\prime$ to a scalar $\mu \in \mfo^\times$, we clearly have $\mu \notin N_E(E^\times)$, and the first Tits construction $J^\prime := J(E,\mu)$ by Example~\ref{e.UNRAAS} is an unramified Freudenthal division $F$-algebra of dimensio $9$ having $\overline{J^\prime} = \bar J$. Hence Thm.~\ref{t.CHADIS} implies that $J \cong J^\prime$ is a first Tits construction, forcing the involution $\tau$ to be distinguished (Thm.~\ref{t.CHADIS}). Conversely, assume $\tau$ is distinguished. Combining \ref{ss.ISDINV} with \cite[Thm.~2.10]{MR2063796}, we find an element $p \in \mfo_J^\times$ such that $\overline{\tau^p} = \bar\tau^{\bar p}$ is a distinguished involution of $\bar B$. By what we have just proved, therefore, $\tau^p$ is a distinguished involution of $B$. But distinguished involutions are essentially unique, and we conclude $(B,\tau) \cong (B,\tau^p)$, hence $(\bar B,\bar\tau) \cong (\bar B,\bar\tau^{\bar p})$. Thus $\bar\tau$ is distinguished.
\end{proof}	

Combining Cor.~\ref{c.REDIIN} with Prop.~\ref{p.HOMDIS}, we finally obtain

\begin{cor} \label{c.UNFRHO} 
An unramified Freudenthal division $F$-algebra $J$ of dimension $9$ is homogeneous if and only if $\bar J$ is homogeneous. \hfill $\square$	
\end{cor}

The analogue of this result for Albert division algebras may be found in \cite[Thm.~6]{MR0379620}. 

\section{Epilogue: Jordan algebras of Clifford type}\label{quadratic}
In this final section, we describe those Jordan algebras of Clifford type that are homogeneous. We begin by recalling the basic definitions.

\subsection{Jordan algebra of a pointed quadratic module} \label{ss.JOQUAM}
Let $(V, q, e)$ be a pointed quadratic module over $k$, 
so $V$ is a finite dimensional $k$-vector space, $q:V\rightarrow k$ a quadratic form and $e\in V$ is a distinguished element, the \emph{base point}, satisfying $q(e)=1$. The conjugation on $V$ is defined by $x\mapsto\overline{x}:= q(e,x)e-x$. 
Then $V$ with the unit element $e$ and the $U$-operator defined by 
$$U_x(y) := q(x, \overline{y})x-q(x)\overline{y},$$
is a Jordan algebra over $k$, denoted by $J := J(V,q,e)$ and called the \emph{Jordan algebra associated with $(V,q,e)$}. If $\dim_k(V) > 1$, then the generic norm of $J$ equals $q$, and it follows from \cite[Exc.~31.34~(c)]{MR4806163} that the structure group of $J$ is the group $GO(q)=\Sim(q)$ of similitudes of $q$. Jordan algebras associated with pointed quadratic modules are said to be \emph{of Clifford type}.

\subsection{Round quadratic forms} \label{ss.RDQUAF} Recall from \cite[p.~52]{MR2427530} that a quadratic form $q$ is said to be \emph{round} if $D(q)=G(q)$. Here 
$D(q)$ is the set of nonzero values of $q$ and $G(q)$ denotes the subgroup of $k^\times$ of similitudes of $q$. For example, Pfister forms, hyperbolic forms are round. 

We now prove

\begin{thm}\label{quad-homog} Let $(V,q,e)$ be a pointed quadratic module over $k$ and $J := J(V,q,e)$ the associated Jordan algebra. If $J$ is homogeneous, then $q$ is round. Conversely, if $q$ is round and regular, then $J$ is homogeneous.
\end{thm}

\begin{proof} Assume $J$ is homogeneous. Since $1 = q(e) \in D(q)$, we have  $G(q) \subseteq D(q)$ \cite[Lemma~9.1]{MR2427530}. Conversely, let $\alpha\in D(q)$. Then $\alpha=q(v)$ for some $v\in V$. Since $\Str(J)$ acts transitively on the invertible elements of $J$, there exists $\psi\in \Str(J)$ with $\psi(v)=e$. Let $\nu(\psi)\in k^{\times}$ be the factor of similitude for $\psi$. Then, taking norms, we get 
$$\nu(\psi)\alpha = \nu(\psi)q(v) = q(\psi(v)) = q(e) = 1, $$
hence $\alpha = \nu(\psi)^{-1} = \nu(\psi^{-1}) \in G(q)$.
Therefore we have $G(q)=D(q)$ and $q$ is round. 

Conversely suppose $q$ is round and regular, so $D(q)=G(q)$ and let $v\in V$ with $\alpha := q(v)\neq 0$. We then have $\psi\in G(q)$ with 
$\nu(\psi)=\alpha$. Hence $q(\psi(e))=\alpha$. By Witt's extension theorem \cite[Thm.~8.3]{MR2427530}, there exists an isometry $\phi:(V,q)\rightarrow (V,q)$ with $\phi(v) = \psi(e)$. Hence $\theta:= \psi^{-1}\phi\in \Str(J)$ satisfies $\theta(v) = e$. Hence $\Str(J)$ is transitive on the set of invertible elements of $J$ and $J$ is homogeneous. 
\end{proof}

Since Pfister quadratic forms are not only round but also stable under base field extensions, the following conclusion is immediate.

\begin{cor} \label{PFIHOM}
Let $q\:V \to k$ be a Pfister quadratic form and suppose $e \in V$ satisfies $q(e) = 1$. Then the Jordan algebra $J(V,q,e)$ is strictly homogeneous. \hfill $\square$    
\end{cor}


\bibliographystyle{amsalpha}
\def\cprime{$'$}
\providecommand{\bysame}{\leavevmode\hbox to3em{\hrulefill}\thinspace}
\providecommand{\MR}{\relax\ifhmode\unskip\space\fi MR }
\providecommand{\MRhref}[2]{%
  \href{http://www.ams.org/mathscinet-getitem?mr=#1}{#2}
}


\end{document}